\documentclass[a4paper,11pt]{article}
\usepackage{amsfonts}
\usepackage{amsmath}
\usepackage{amssymb}
\usepackage{amsthm}
\usepackage[mathscr]{eucal}
\usepackage{latexsym}
\usepackage[dvips]{graphicx}
\usepackage{psfrag}
\usepackage{a4wide}
\usepackage{color}

\theoremstyle{plain}
 \newtheorem{theorem}{Theorem}[section]
 \newtheorem*{theorem*}{Theorem}
 \newtheorem{proposition}[theorem]{Proposition}
 \newtheorem{lemma}[theorem]{Lemma}
 
 \newtheorem{problem}[theorem]{Problem}
 
\theoremstyle{definition}
 \newtheorem{definition}[theorem]{Definition}
\theoremstyle{remark}
 
 \newtheorem{example}[theorem]{Example}
\numberwithin{equation}{section}

\begin{document}
\title{$Q$-polynomial coherent configurations}

\author{
Sho Suda \\ Department of Mathematics, \\ National Defense Academy of Japan \\ ssuda@nda.ac.jp
}
\date{\today}

\maketitle
\renewcommand{\thefootnote}{\fnsymbol{footnote}}
\begin{abstract} 
Coherent configurations are a generalization of association schemes. 
In this paper, we introduce the concept of $Q$-polynomial coherent configurations and study the relationship among intersection numbers, Krein numbers, and eigenmatrices. 
The examples of $Q$-polynomial coherent configurations are provided from Delsarte designs in $Q$-polynomial schemes and spherical designs.  
\end{abstract}



\section{Introduction}
Association schemes are a combinatorial generalization of a transitive permutation group.  
$Q$-polynomial association schemes are defined by Delsarte in \cite{D} as a framework to study design theory including orthogonal arrays and block designs, and
have been extensively studied in the last two decades.  

This concept is regarded as a dual object to distance-regular graphs (equivalently $P$-polynomial association schemes).
Many examples of $Q$-polynomial association schemes that are neither $P$-polynomial nor duals of translation $P$-polynomial association schemes are obtained from designs in $Q$-polynomial schemes or spherical designs \cite{BB,DGS}. 

Coherent configurations are a combinatorial generalization of a permutation group. 
In the last decades, several examples of coherent configurations are obtained from design theoretic objects such as block designs, spherical designs and Euclidean designs. 

In this paper, the $Q$-polynomial property for coherent configurations whose fibers are symmetric association schemes is proposed.  
The $Q$-polynomial property is characterized in a similar fashion to association schemes.  
Examples will be given from Delsarte designs or spherical designs. 
It was shown in \cite{D,DGS} that a Delsarte or spherical $t$-design with degree $s$ satisfying $t\geq2s-2$ has a structure of a $Q$-polynomial association scheme. 
In \cite{S}, this result for spherical designs is generalized as follows: 
Let $X_i$ be a spherical $t_i$-design for $i\in\{1,2,\ldots,n\}$ and $s_{i,j}$ be the number of distinct inner products between $X_i$ and $X_j$. If $t_j\geq s_{i,j}+s_{j,h}-2$ holds for any $i,j,h\in\{1,2,\ldots,n\}$, then $\bigcup_{i=1}^n X_i$ with binary relations defined by  inner products has a structure of a coherent configuration.  
We show that the coherent configurations obtained in this manner is $Q$-polynomial.     
As a corolalry, $Q$-polynomial coherent configurations are obtained from 
\begin{itemize}
\item tight Delsarte or spherical designs with small strength such as $4,5,7$,  
\item $Q$-antipodal  $Q$-polynomial association schemes.    
\end{itemize}

This paper is organized as follows. 
In Section~\ref{sec:as} and Section~\ref{sec:cc}, we review the theory of association schemes and coherent configurations.  
In Section~\ref{sec:qcc}, we prove Proposition~\ref{prop:Q-poly} that characterizes the $Q$-polynomial property. 
In Section~\ref{sec:example}, several examples of $Q$-polynomial coherent configurations are provided from $Q$-antipodal $Q$-polynomial association schemes, complete orthogonal array of strength $4$, tight spherical $t$-designs for $t=4,5,7$, and maximal mutually unbiased bases. 
Section~\ref{sec:hn2} is taken from \cite{V} and \cite{BBTY2020}. It is known that the Terwilliger algebra of the binary Hamming schemes is a coherent configuration. We will claim that the coherent configuration is $Q$-polynomial based on \cite{V}. Furthermore it was shown in \cite{BBTY2020} that tight relative $2e$-designs on two shells in the binary Hamming scheme $H(n,2)$ yield a $Q$-polynomial coherent configurations. Motivated by this work, we generalize Theorem~\ref{thm:Q-polyccDel} to designs in fibers of a  $Q$-polynomial coherent configuration. 
Finally we list open problems in Section~\ref{sec:open} regarding $Q$-polynomial coherent configurations. In Appendix~\ref{sec:a}, the formula among intersection numbers, Krein numbers, eigenmatrices is given in a similar manner \cite{BI}.

\section{Association schemes}\label{sec:as}

We begin with the definition of association schemes.
We refer the reader to \cite{BI} for more information. 
Let $X$ be a finite set and $\mathcal{R}=\{R_0,R_1,\ldots,R_d\}$ be a set of non-empty subsets of $X\times X$.

The pair $\mathfrak{X}=(X,\mathcal{R})$ is a commutative association scheme with class $d$ if the following hold:
\begin{enumerate}
\item $R_0=\text{diag}(X\times X)$, where $\mathrm{diag}(X\times X)=\{(x,x)\mid x\in X\}$, 
\item $\{R_0,R_1,\ldots,R_d\}$ is a partition of $X\times X$,  
\item for any $i \in\{1,\ldots,d\}$, $R_i^\top\in \mathcal{R}$ for $1\leq i\leq d$, where $R^\top=\{(y,x)\mid (x,y)\in R\}$ for a subset $R$ of $X\times X$, 
\item for any $i,j,h \in\{0,1,\ldots,d\}$, there exists an integer $p_{i,j}^h$, called an intersection number,  such that 
$$
|\{z\in X \mid (x,y)\in R_i,(y,z)\in R_j\}|=p_{i,j}^h
$$
for any $(x,y)\in R_h$, 
\item for any $i,j,h \in\{0,1,\ldots,d\}$, $p_{i,j}^h=p_{j,i}^h$ holds.  
\end{enumerate}

A commutative association scheme is said to be symmetric if the following holds: 
\begin{enumerate}
\item[(3)'] for any $i\in\{1,2,\ldots,d\}$, $R_i^\top=R_i$ holds. 
\end{enumerate}

From now, let $\mathfrak{X}=(X,\mathcal{R})$ be a symmetric association scheme. 
Let $A_i$ be the adjacency matrix of the graph $(X,R_i)$. 
Here the adjacency matrix of a graph $(X,R)$ is the $(0,1)$-matrix with rows and columns indexed by the elements of $X$ and its $(x,y)$-entry equal to $1$  if $(x,y)\in R$ and $0$ otherwise. 
The vector space $\mathcal{A}$ spanned by the $A_i$ over $\mathbb{R}$ forms an algebra which is called the adjacency algebra of $(X,\mathcal{R})$.
Since $\mathcal{A}$ is commutative and semisimple, there exist primitive idempotents $E_0=\frac{1}{|X|}J,E_1,\ldots,E_d$, where $J$ is the all-ones matrix. 
Since the adjacency algebra $\mathcal{A}$ is closed under the ordinary multiplication and entry-wise multiplication denoted by $\circ$, 
reformulate the intersection numbers $p_{i,j}^h$ and define the Krein numbers $q_{i,j}^h$ for $0\leq i,j,h\leq d$ as follows;
\begin{align*}
A_iA_j=\sum_{h=0}^d p_{i,j}^h A_h,\quad E_i\circ E_j=\frac{1}{|X|}\sum_{h=0}^d q_{i,j}^h E_h.
\end{align*}
We define Krein matrices $\hat{B}_i=(q_{\ell,j}^h)_{j,h=0}^{d}$ for $i\in\{0,1,\ldots,d\}$.

Since $\{A_0,A_1,\ldots,A_d\}$ and $\{E_0,E_1,\ldots,E_d\}$ form bases of $\mathcal{A}$, there exist change of bases matrices $P=(p_h(\ell))_{0\leq \ell,j\leq d}$ and $Q=(q_h(\ell))_{0\leq \ell,j\leq d}$ defined by 
\begin{align*}
(A_0,A_1,\ldots,A_d)&=(E_0,E_1,\ldots,E_d)P, \\
(E_0,E_1,\ldots,E_d)&=\frac{1}{|X|}(A_0,A_1,\ldots,A_d)Q, 
\end{align*}
equivalently, 
\begin{align*}
A_h=\sum_{\ell=0}^d p_h(\ell)E_\ell,\quad E_h=\frac{1}{|X|}\sum_{\ell=0}^d q_h(\ell)A_\ell.
\end{align*}
The matrices $P$ and $Q$ are called the first and second eigenmatrices of $(X,\mathcal{R})$ respectively.

\begin{proposition}\label{prop:Q-polyas}
Let $\mathfrak{X}=(X,\mathcal{R})$ be a symmetric association scheme.
The following conditions are equivalent:
\begin{enumerate}
\item there exists a set of polynomials $\{v_h(x)\mid 0\leq h \leq d \}$ satisfying that for any $h\in\{0,1,\ldots,d\}$, $\mathrm{deg}v_h(x)=h$ and $|X|E_h=v_h(|X|E_1)$ under the entry-wise product,
\item there exists a set of polynomials $\{v_h(x)\mid 0\leq h \leq d \}$ satisfying that for any $h,\ell\in\{0,1,\ldots,d\}$, $\mathrm{deg}v_h(x)=h$ and $q_h(\ell)=v_h(\theta_\ell^*)$, where $\theta_\ell^*=q_1(\ell)$,
\item The Krein matrix $\hat{B}_1=(q_{1,j}^h)_{j,h=0}^{d}$ is a tridiagonal matrix with non-zero superdiagonal and subdiagonal entries. 
\end{enumerate}
\end{proposition}
\begin{proof}
See \cite[pp.193-194]{BI}. 
\end{proof}

The symmetric association scheme  $(X,\mathcal{R})$ is said to be $Q$-polynomial 
if one of the conditions in Proposition~\ref{prop:Q-polyas} holds.
For a $Q$-polynomial association scheme, set $a_i^*=q_{1,i}^i$ ($i\in\{0,1,\ldots,d\}$), $b_i^*=q_{1,i+1}^{i}$ ($i\in\{0,1,\ldots,d-1\}$), $c_i^*=q_{1,i-1}^i$ ($i\in\{1,2,\ldots,d\}$).

\section{Coherent configurations}\label{sec:cc}

Let $X$ be a non-empty finite set. 
For a subset $R$ of $X\times X$, define the projection of $R$ as follows:
\begin{align*}
\mathrm{pr}_1(R)&=\{x\in X\mid (x,y)\in R \text{ for some }y\in X\},\\
\mathrm{pr}_2(R)&=\{y\in X\mid (x,y)\in R \text{ for some }x\in X\}.
\end{align*}

\begin{definition}\label{cc}
Let $X$ be a non-empty finite set and $\mathcal{R}=\{R_i \mid i\in I\}$ be a set of non-empty subsets of $X\times X$.
The pair $\mathcal{C}=(X,\mathcal{R})$ is a coherent configuration if the following properties are satisfied:
\begin{enumerate}
\item $\{R_i\}_{i\in I}$ is a partition of $X\times X$,
\item for any $i \in I$, $R_i^\top\in \mathcal{R}$, 
\item $R_i \cap \mathrm{diag}(X\times X)\neq \emptyset$ implies $R_i \subset\mathrm{diag}(X\times X)$,
\item for any $i,j,h\in I$, the number $|\{z\in X \mid (x,z)\in R_i,(z,y)\in R_j\}|$ is independent of the choice of $(x,y)\in R_h$.   
\end{enumerate}
\end{definition}
Let $A_i$ be the adjacency matrix of the graph $(X,R_i)$ for $i\in I$. 
We define the coherent algebra $\mathcal{A}$ of the coherent configuration $\mathcal{C}$ as the subalgebra of $\mathrm{Mat}_{|X|}(\mathbb{C})$ generated by $\{A_i \mid i\in I\}$ over $\mathbb{C}$. 
There uniquely exists a subset $\Omega$ in $I$ such that $\mathrm{diag}(X\times X)=\bigcup_{i\in \Omega}R_i$ by Definition~\ref{cc}(1) and (3).
We obtain the standard partition $\{X_i\}_{i\in \Omega}$ of $X$ where $X_i=\mathrm{pr}_1(R_i)=\mathrm{pr}_2(R_i)$ for $i\in\Omega$. 
We call $X_i$ a fiber of the coherent configuration $\mathcal{C}$.  
The following property of binary relations of coherent configurations was shown in \cite{Hig2}:
\begin{lemma}\label{relation}
For any $i\in I$, there exist $j,h\in \Omega$ such that $\mathrm{pr}_1(R_i)=X_j$, $\mathrm{pr}_2(R_i)=X_h$.
\end{lemma}

For $i,j\in \Omega$, define  
$I^{(i,j)}=\{R_\ell \mid \ell\in I, R_\ell \subset X_i\times X_j\}$. 
Lemma~\ref{relation} implies that $\{I^{(i,j)}\mid i,j\in\Omega\}$ is a partition of $I$.
We put $r_{i,j}=|I^{(i,j)}|-\delta_{i,j}$, and we call the matrix $(|I^{(i,j)}|)_{i,j\in\Omega}$ the type of the coherent configuration $\mathcal{C}$. 

Let $\varepsilon_{i,j}=1-\delta_{i,j}$.
By the partition $\{I^{(i,j)}\mid i,j\in\Omega\}$ of $I$, the elements of $I^{(i,j)}$ are renumbered as $R_{\varepsilon_{i,j}}^{(i,j)},\ldots,R_{r_{i,j}}^{(i,j)}$ such that $R_0^{(i,i)}=\mathrm{diag}(X_i\times X_i)$ and $(R_\ell^{(i,j)})^\top =R_\ell^{(j,i)}$. 
We denote the adjacency matrix of $R_\ell^{(i,j)}$ as $A_\ell^{(i,j)}$. 
For $i,j\in \Omega$, define by $\mathcal{A}^{(i,j)}$ the vector space spanned by  $A_\ell^{(i,j)}$ ($\varepsilon_{i,j}\leq \ell \leq r_{i,j}$) over $\mathbb{C}$. 
Then $\mathcal{A}^{(i,j)}\mathcal{A}^{(j,h)}\subset\mathcal{A}^{(i,h)}$ holds and  define intersection numbers $p_{\ell,m,n}^{(i,j,h)}$ as 
$$
A_\ell^{(i,j)}A_m^{(j,h)}=\sum\limits_{n=\varepsilon_{i,j}}^{r_{i,j}}p_{\ell,m,n}^{(i,j,h)}A_n^{(i,h)}.
$$
Set $k_\ell^{(i,j)}=p_{\ell,\ell,0}^{(i,j,i)}$.
Then $k_\ell^{(i,j)}=|\{y\in X_j\mid (x,y)\in R_\ell^{(i,j)}\}|$ for any $x\in X_i$.
We call $k_\ell^{(i,j)}$ the valency of $R_\ell^{(i,j)}$.

Let $\tilde{r}_{i,j}=r_{i,j}-\varepsilon_{i,j}$. 
Ito and Munemasa proved in \cite{IM} that if the fiber $\mathcal{C}^i=(X_i,I^{(i,i)})$ is a commutative association scheme, then there exists a basis $\{\varepsilon_{i,j}^{s}\mid s\in S, i,j\in F_s\}$ of $\mathcal{A}$ such that 
\begin{itemize}
\item $\{\varepsilon_{i,j}^{s}\mid i,j\in F_s\}$ ($s\in S$) generates a simple two-sided ideal $\mathfrak{C}_s$ of $\mathcal{A}$ with $\mathcal{A}=\oplus_{s\in S}\mathfrak{C}_s$, 
\item for any $s\in S$, $\varepsilon_{i,j}^{s}\varepsilon_{k,\ell}^{s}=\delta_{j,k}\varepsilon_{i,\ell}^{s}$ holds, 
\item for any $s\in S$, $(\varepsilon_{i,j}^{s})^*=\varepsilon_{j,i}^{s}$ holds, 
\item for any $s\in S$ and $i,j\in F_s$, $\varepsilon_{i,j}^{s}\in\bigcup_{k,\ell\in\Omega} \mathcal{A}^{(k,\ell)}$, 
\item for any $s\in S$ and $i,j\in F_s$, $\dim(\mathfrak{C}_s\cap \mathcal{A}^{(k,\ell)})\leq 1$.
\end{itemize}

In this paper we consider coherent configurations $\mathcal{C}$ such that the fiber $\mathcal{C}^i=(X_i,I^{(i,i)})$ is a symmetric association scheme 
for any $i\in \Omega$ and there exists a basis $\{E_\ell^{(i,j)}\mid i,j\in\Omega, 0\leq \ell\leq \tilde{r}_{i,j}\}$ of $\mathcal{A}$ such that 
\begin{enumerate}
\item[(B1)] for any $i,j\in\Omega$, $E_0^{(i,j)}=\frac{1}{\sqrt{|X_i||X_j|}}J_{|X_i|,|X_j|}$, where $J_{p,q}$ is the $p\times q$ all-ones matrix, \label{b1}
\item[(B2)] for any $i,j\in\Omega$, $\{E_\ell^{(i,j)}\mid 0\leq \ell \leq \tilde{r}_{i,j}\}$ is a basis of $\mathcal{A}^{(i,j)}$ as a vector space, 
\item[(B3)] for any $i,j\in\Omega,\ell\in \{0,1,\ldots,\tilde{r}_{i,j}\}$, $(E_\ell^{(i,j)})^\top=E_\ell^{(j,i)}$,
\item[(B4)] for any $i,j,i',j'\in\Omega$ and $\ell\in \{0,1,\ldots,\tilde{r}_{i,j}\}$, $\ell'\in \{0,1,\ldots,\tilde{r}_{i',j'}\}$, $E_\ell^{(i,j)}E_{\ell'}^{(i',j')}=\delta_{\ell,\ell'}\delta_{j,i'}E_\ell^{(i,j')}$.
\end{enumerate}

Since $\mathcal{A}^{(i,j)}$ is closed under the entry-wise product $\circ$, 
we define Krein parameters $q_{\ell,m,n}^{(i,j)}$  as follows:
$$
E_\ell^{(i,j)}\circ E_m^{(i,j)}=\frac{1}{\sqrt{|X_i||X_j|}}\sum\limits_{n=0}^{\tilde{r}_{i,j}}q_{\ell,m,n}^{(i,j)}E_n^{(i,j)}.
$$
We define Krein matrices $\hat{B}_\ell^{(i,j)}=(q_{\ell,m,n}^{(i,j)})_{m,n=0}^{\tilde{r}_{i,j}}$ for $\ell\in\{0,1,\ldots,\tilde{r}_{i,j}\}$.
 
For $i,j\in\Omega$, since $\{A_\ell^{(i,j)}\mid \varepsilon_{i,j}\leq \ell\leq r_{i,j}\}$ and $\{E_\ell^{(i,j)}\mid 0\leq \ell\leq \tilde{r}_{i,j}\}$ are bases of $\mathcal{A}^{(i,j)}$, there exist change-of-bases matrices  $P^{(i,j)}=(p_{h}^{(i,j)}(\ell))_{\substack{0\leq \ell\leq \tilde{r}_{i,j} \\ \varepsilon_{i,j}\leq h\leq r_{i,j} }}$, $Q^{(i,j)}=(q_h^{(i,j)}(\ell))_{\substack{\varepsilon_{i,j}\leq \ell\leq r_{i,j}\\ 0\leq h\leq \tilde{r}_{i,j}}}$ such that  
\begin{align*}
(A_{\varepsilon_{i,j}}^{(i,j)},\ldots,A_{r_{i,j}}^{(i,j)})&=(E_0^{(i,j)},\ldots,E_{\tilde{r}_{i,j}}^{(i,j)})P^{(i,j)}, \\
(E_0^{(i,j)},\ldots,E_{\tilde{r}_{i,j}}^{(i,j)})&=\frac{1}{\sqrt{|X_i||X_j|}}(A_{\varepsilon_{i,j}}^{(i,j)},\ldots,A_{r_{i,j}}^{(i,j)})Q^{(i,j)}, 
\end{align*}
equivalently, 
\begin{align*}
A_h^{(i,j)}=\sum_{\ell=0}^{\tilde{r}_{i,j}} p_h^{(i,j)}(\ell)E_\ell^{(i,j)},\quad E_h^{(i,j)}=\frac{1}{|X|}\sum_{\ell=\varepsilon_{i,j}}^{r_{i,j}} q_h^{(i,j)}(\ell)A_\ell^{(i,j)}.
\end{align*}
We will show several relations among $p_{\ell,m,n}^{(i,j,h)}$, $q_{\ell,m,n}^{(i,j)}$, $p_{h}^{(i,j)}(\ell)$, $q_{h}^{(i,j)}(\ell)$ in Appendix A as in the case of symmetric association schemes.  

\section{$Q$-polynomial properties of coherent configurations}\label{sec:qcc}

The following proposition characterizes $Q$-polynomial property.
\begin{proposition}\label{prop:Q-poly}
Let $\mathcal{C}$ be a coherent configuration such that each fiber is a symmetric association scheme and there exists a basis $\{E_\ell^{(i,j)}\mid i,j\in\Omega, 0\leq \ell\leq \tilde{r}_{i,j}\}$ of $\mathcal{A}$ satisfying (B1)-(B4).
The following conditions are equivalent:
\begin{enumerate}
\item for any $i,j\in\Omega$, there exists a set of polynomials $\{v_h^{(i,j)}(x)\mid  0\leq h\leq \tilde{r}_{i,j}\}$ satisfying that for any $h\in\{0,1,\ldots,\tilde{r}_{i,j}\}$, $\mathrm{deg}v_h^{(i,j)}(x)=h$ and $\sqrt{|X_i||X_j|}E_h^{(i,j)}=v_h^{(i,j)}(\sqrt{|X_i||X_j|}E_1^{(i,j)})$ under the entry-wise product,
\item for any $i,j\in\Omega$, there exists a set of polynomials $\{v_h^{(i,j)}(x)\mid  0\leq h\leq \tilde{r}_{i,j}\}$ satisfying that for any $h,\ell\in\{0,1,\ldots,\tilde{r}_{i,j}\}$, $\mathrm{deg}v_h^{(i,j)}(x)=h$ and $q_h^{(i,j)}(\ell)=v_h^{(i,j)}(\theta_\ell^{(i,j)})$, where $\theta_\ell^{(i,j)}=q_1^{(i,j)}(\ell)$,
\item for any $i,j\in\Omega$, the Krein matrix $\hat{B}_1^{(i,j)}$ is a tridiagonal matrix with non-zero superdiagonal
and subdiagonal entries. 
\end{enumerate}
\begin{proof}
$(1)\Leftrightarrow(2)$:
Putting $\theta_\ell^{(i,j)}=q_1^{(i,j)}(\ell)$, we have $E_1^{(i,j)}=\frac{1}{\sqrt{|X_i||X_j|}}\sum\limits_{\ell=\varepsilon_{i,j}}^{r_{i,j}}\theta_\ell^{(i,j)}A_\ell^{(i,j)}$.
Suppose $(1)$ holds.
Then $\sqrt{|X_i||X_j|}E_h^{(i,j)}=v_h^{(i,j)}(\sqrt{|X_i||X_j|}E_1^{(i,j)})=\sum\limits_{\ell=\varepsilon_{i,j}}^{r_{i,j}}v_h^{(i,j)}(\theta_\ell^{(i,j)})A_\ell^{(i,j)}$, 
so we obtain $q_h^{(i,j)}(\ell)=v_k^{(i,j)}(\theta_\ell^{(i,j)})$.
Conversely suppose $(2)$ holds.
Then $\sqrt{|X_i||X_j|}E_h^{(i,j)}=\sum\limits_{\ell=\varepsilon_{i,j}}^{r_{i,j}}q_\ell^{(i,j)}(h)A_\ell^{(i,j)}=\sum\limits_{\ell=\varepsilon_{i,j}}^{r_{i,j}}v_h^{(i,j)}(\theta_\ell^{(i,j)})A_\ell^{(i,j)}=v_h^{(i,j)}(\sqrt{|X_i||X_j|}E_1^{(i,j)})$.

$(2)\Rightarrow(3)$:
Suppose $(2)$ holds.
Since the polynomial $xv_h^{(i,j)}(x)$ can be written as a linear combination of $v_{h+1}^{(i,j)}(x),v_h^{(i,j)}(x),\ldots,v_0^{(i,j)}(x)$, 
$E_h^{(i,j)}\circ E_1^{(i,j)}$ is a linear combination of $E_{h+1}^{(i,j)},E_h^{(i,j)},\ldots,E_0^{(i,j)}$.
Therefore $q_{1,h,\ell}^{(i,j)}=0$ if $\ell\geq h+2$, and $q_{1,h,h+1}^{(i,j)}\neq0$.
By Proposition~\ref{6} (5), we obtain $m_\ell^{(i,j)}q_{1,h,\ell}^{(i,j)}=m_h^{(i,i)}q_{1,\ell,h}^{(i,j)}$.
Therefore $q_{1,h,\ell}^{(i,j)}=0$ if and only if $q_{1,\ell,h}^{(i,j)}=0$.
Hence $q_{1,h,\ell}^{(i,j)}=0$ if $\ell\leq h-2$, and $q_{1,h,h-1}^{(i,j)}\neq 0$.

$(3)\Rightarrow(1)$: 
Suppose $(3)$ holds.
Set $b_h^{(i,j)}=q_{1,h+1,h}^{(i,j)}$, $a_h^{(i,j)}=q_{1,h,h}^{(i,j)}$, $c_h^{(i,j)}=q_{1,h,h+1}^{(i,j)}$.
Since $E_1^{(i,j)}\circ E_h^{(i,j)}=\frac{1}{\sqrt{|X_i||X_j|}}\sum\limits_{\alpha=0}^{\tilde{r}_{i,j}}q_{1,h,\alpha}^{(i,j)}E_\alpha^{(i,j)}$, 
$E_1^{(i,j)}\circ E_h^{(i,j)}=b_{h-1}^{(i,j)}E_{h-1}^{(i,j)}+a_h^{(i,j)}E_h^{(i,j)}+c_{h+1}^{(i,j)}E_{h+1}^{(i,j)}$.
We define $v_0^{(i,j)}(x)=1$, $v_1^{(i,j)}(x)=x$ and polynomials $v_h^{(i,j)}(x)$ of degree $h$ as recurrence 
$$xv_h^{(i,j)}(x)=b_{h-1}^{(i,j)}v_{h-1}^{(i,j)}(x)+a_h^{(i,j)}v_h^{(i,j)}(x)+c_{h+1}^{(i,j)}v_{h+1}^{(i,j)}(x) .$$
Then $v_h^{(i,j)}(\sqrt{|X_i||X_j|}E_1^{(i,j)})=\sqrt{|X_i||X_j|}E_h^{(i,j)}$.
\end{proof}
\end{proposition}  
\begin{definition}\label{Q-polycc}
Let $\mathcal{C}$ be a coherent configuration such that each fiber is a symmetric association scheme and there exists a basis $\{E_\ell^{(i,j)}\mid i,j\in\Omega, 0\leq \ell\leq \tilde{r}_{i,j}\}$ of $\mathcal{A}$ satisfying (B1)-(B4). 
The coherent configuration $\mathcal{C}$ is said to be $Q$-polynomial if one of $(1)$--$(3)$ of Proposition~\ref{prop:Q-poly} holds. 
\end{definition}

\section{Examples of $Q$-polynomial coherent configurations}\label{sec:example}
\subsection{An $n$-th power of a $Q$-polynomial association scheme}\label{subsec:pq}
We introduce an $n$-th power of a symmetric association scheme for a positive integer $n\geq 2$. 
Let $\mathfrak{X}=(X,\{R_i\}_{i=0}^d)$ be a symmetric association scheme with primitive idempotents $E_0,E_1,\ldots,E_d$. 
We define a coherent configuration $\mathcal{C}_n=(\coprod_{i=1}^n X_i,\{R_\ell^{(i,j)} \mid i,j\in\{1,2,\ldots,n\}, \ell\in\{0,1,\ldots,d\}\})$ where $X_i=X$ and $\coprod_{i=1}^n X_i$ is a disjoint union of $X_i$'s, and $R_\ell^{(i,j)}=\{(x,y)\in (\coprod_{i=1}^n X_i)^2 \mid x\in X_i,y\in X_j,(x,y)\in R_\ell\}$\footnote{The index $\ell$ of $R_\ell^{(i,j)}$ should start with $1$ when $i\neq j$, but we avoid it.}. We call $\mathcal{C}_n$ an $n$-th power of $\mathfrak{X}$. 

For $i,j\in\{1,2,\ldots,n\}$ and $\ell\in\{0,1,\ldots,d\}$, define $E_\ell^{(i,j)}=e_{i,j}\otimes E_\ell$ where $e_{i,j}$ denotes the $n\times n$ matrix with a $1$ in the $(i,j)$-entry and $0$'s elsewhere.  Then $\{E_\ell^{(i,j)}\mid i,j\in\{1,2,\ldots,n\}, \ell\in\{0,1,\ldots,d\}\}$ is a basis of the coherent algebra of $\mathcal{C}_n$ satisfying (B1)--(B4).
\begin{example}
Let $n$ be a positive integer at least two and  $\mathfrak{X}=(X,\{R_\ell\}_{\ell=0}^d)$ a $Q$-polynomial association scheme with respect to the ordering of the primitive idempotents $E_0,E_1,\ldots,E_d$. 
Then the coherent configuration $\mathcal{C}_n$ is a $Q$-polynomial coherent configuration. 
\end{example}

\subsection{Delsarte designs of a $Q$-polynomial association scheme}

Let $\mathfrak{X}=(X,\{R_i\}_{i=0}^d)$ be a $Q$-polynomial association scheme with respect to the ordering of the primitive idempotents $E_0,E_1,\ldots,E_d$.  
For a non-empty subset $C$ in $X$, we define the characteristic vector $\chi=\chi_C$ as a column vector indexed by $X$ whose $x$-th entry is $1$ if $x\in C$, and $0$ otherwise. 
For a positive integer $t$, a subset $C$ is said to be a (Delsarte) $t$-design if $E_i \chi=0$ for any $i\in\{1,\ldots,t\}$.  

Define a real numbers $b_i$ ($i\in\{0,1,\ldots,d\}$) by $b_i=\frac{|X|}{|Y|}\chi^\top E_i \chi$.
Note that the numbers $b_i$ are non-negative because $E_i$ is positive semidifinite. 
A subset $C$ is a $t$-design if and only if $b_1=\cdots=b_t=0$.  
The vector $(b_i)_{i=0}^d$ is said to be the dual inner distribution of $C$. 
Designs in the Hamming schemes or Johnson scheme are characterized by orthogonal arrays or block designs. 

\begin{example} 
An orthogonal array $\mathrm{OA}_\lambda(t,n,q)$ is a $\lambda q^t\times n$ matrix over an alphabet of size $q$ in which each set of $t$ columns contains each $t$-tuples over the alphabet exactly $\lambda$ times as a row.  
An orthogonal array $\mathrm{OA}_\lambda(t,n,q)$ is a $t$-design in a Hamming scheme $H(n,q)$ with respect to the ordering of the primitive idempotents determined from $b_i^*=(n-i)(q-1),c_i^*=i$. 
\end{example}

\begin{example} 
A $t$-$(v,k,\lambda)$ design is a collection of $k$-subsets (called blocks) of a $v$-set such that every $t$-subset is contained in exactly $\lambda$ blocks. 
A  $t$-$(v,k,\lambda)$ design is a $t$-design in a Johnson scheme $J(v,k)$ with respect to the ordering of the primitive idempotents determined from $b_i^*=\frac{v(v-1)(v-i+1)(v-k-i)(k-i)}{k(v-k)(v-2i+1)(v-2i)},c_i^*=\frac{v(v-1)i(k-i+1)(v-k-i+1)}{k(v-k)(v-2i+2)(v-2i+1)}$. 
\end{example}

For a subset $C$, define 
\begin{align*}
A(C)&=\{\ell \mid 1\leq \ell \leq d,R_\ell \cap(C\times C)\neq \emptyset\}. 
\end{align*}
Let $s=|A(C)|$ and we call $s$ the degree of $C$. 
Let  
 $A(C)=\{\alpha_1,\alpha_2,\ldots,\alpha_{s}\}$ and set $\alpha_0=0$.  
For $\ell\in\{0,1,\ldots,s\}$, define a subset $R_\ell$ of $C\times C$ by 
$$
R_\ell=\{(x,y)\in C\times C \mid (x,y)\in R_{\alpha_\ell} \}.
$$
The following theorem is due to \cite{D}. 
\begin{theorem}\label{thm:Q-polyDel}
Let $\mathfrak{X}=(X,\{R_i\}_{i=0}^d)$ be a $Q$-polynomial association scheme with respect to the ordering of the primitive idempotents $E_0,E_1,\ldots,E_d$.
Let $C$ be a $t$-design with degree $s$. 
If $2s-2 \leq t$ holds, then $(C, \{R_\ell \mid 0\leq \ell \leq s\})$ is a $Q$-polynomial association scheme.
\end{theorem}
We then generalize Theorem~\ref{thm:Q-polyDel} to disjoint designs in a $Q$-polynomial association scheme.   
Let $X_1,X_2,\ldots,X_n$ be disjoint subsets of $X$.  
Define 
\begin{align*}
A(X_i,X_j)&=\{\ell \mid 1\leq \ell \leq d,R_\ell \cap(X_i\times X_j)\neq \emptyset\}.
\end{align*}
Let $s_{i,j}=|A(X_i,X_j)|$ and  
 $A(X_i,X_j)=\{\alpha_1^{(i,j)},\alpha_2^{(i,j)},\ldots,\alpha_{s_{i,j}}^{(i,j)}\}$ and set $\alpha_0^{(i,i)}=0$.  
For $\ell\in\{\varepsilon_{i,j},1,\ldots,s_{i,j}\}$, define a subset $R_\ell^{(i,j)}$ of $X\times X$ by 
$$
R_\ell^{(i,j)}=\{(x,y)\in X\times X \mid x\in X_i,y\in X_j, (x,y)\in R_{\alpha_\ell^{(i,j)}} \}.
$$
Denote by $A_\ell^{(i,j)}$ the adjacency matrix of the graph $(\bigcup_{i=1}^n X_i,R_\ell^{(i,j)})$. 
For $i\in\{1,2,\ldots,n\}$, define $\Delta_{X_i}$ be the diagonal matrix indexed by the elements of $X$ with $(x,x)$-entry equal to $1$ if $x\in X_i$ and $0$ otherwise, and 
$\tilde{\Delta}_{X_i}$ as the matrix obtained from $\Delta_{X_i}$ by restricting the rows to $\bigcup_{i=1}^n X_i$. 
Note that 
\begin{align}\label{eq:dd}
(\tilde{\Delta}_{X_i})^\top\tilde{\Delta}_{X_i}=\Delta_{X_i}\text{ and }\tilde{\Delta}_{X_i}\Delta_{X_i}=\tilde{\Delta}_{X_i}.
\end{align}
For a real matrix $A$, define $||A||=\sqrt{\mathrm{tr}(AA^\top)}$. 

\begin{theorem}\label{thm:Q-polyccDel}
Let $\mathfrak{X}=(X,\{R_i\}_{i=0}^d)$ be a $Q$-polynomial association scheme with respect to the ordering of the primitive idempotents $E_0,E_1,\ldots,E_d$.
Let $X_i$ be a $t_i$-design for $i \in \{1,2,\dots,n\}$.
Assume that $X_i \cap X_j=\emptyset$ for distinct integers $i ,j \in \{1,2,\dots,n\}$.
If $s_{i,j}+s_{j,h}-2 \leq t_j$ holds for any $i,j,h \in \{1,2,\dots ,n\}$, 
then $(\bigcup_{i=1}^n X_i, \{R_\ell^{(i,j)} \mid 1\leq i,j\leq n,\varepsilon_{i,j}\leq \ell \leq s_{i,j}\})$ is a $Q$-polynomial coherent configuration.
\end{theorem}
\begin{proof}
Let $i,j,h\in \{1,2,\ldots,n\}$, $\alpha\in\{0,1,\ldots,s_{i,j}-1\}$, $\beta\in\{0,1,\ldots, s_{j,h}-1\}$. 
Since $s_{i,j}+s_{j,h}-2\leq t_j$, it holds that $\alpha+\beta+1\leq t_j+1$ and $q_{\alpha,\beta}^\ell=0$ for $\ell\geq t_j+1$. 
Since $X_j$ is a $t_j$-design, the dual distribution $(b_\ell^{(j)})_{\ell=0}^{d}$ of $X_j$ satisfies that $b_\ell^{(j)}=0$ for $\ell\leq t_j$. Then  
\begin{align*}
|||X|E_{\alpha}\Delta_{X_j} E_{\beta}-|X_j|\delta_{\alpha,\beta}E_{\alpha}||^2&=|X_j|\sum\limits_{\ell=1}^d q_{\alpha,\beta}^\ell b_\ell^{(j)} \displaybreak[0]\\
&=|X_j|\left(\sum\limits_{\ell=1}^{t_j}q_{\alpha,\beta}^\ell b_\ell^{(j)}+\sum\limits_{\ell=t_j+1}^d q_{\alpha,\beta}^\ell b_\ell^{(j)}    \right) \\
&=0.
\end{align*}
Therefore 
$$
|X|E_{\alpha}\Delta_{X_j} E_{\beta}=|X_j|\delta_{\alpha,\beta}E_{\alpha}.
$$
Multiplying $\Delta_{X_i}$ on the left side and $\Delta_{X_h}$ on the right side, we obtain
\begin{align}
|X|\Delta_{X_i}E_{\alpha}\Delta_{X_j} E_{\beta}\Delta_{X_h}=|X_j|\delta_{\alpha,\beta}\Delta_{X_i}E_{\alpha}\Delta_{X_h}. \label{eq:11}
\end{align}
Define 
\begin{align*}
A_\ell^{(i,j)}=\tilde{\Delta}_{X_i}A_\ell(\tilde{\Delta}_{X_j})^\top,\quad 
E_{\ell'}^{(i,j)}=\frac{|X|}{\sqrt{|X_i||X_j|}}\tilde{\Delta}_{X_i}E_{\ell'}(\tilde{\Delta}_{X_j})^\top 
\end{align*}
for $i,j\in\{1,2,\ldots,n\}, \ell \in A(X_i,X_j), \ell' \in \{0,1,\ldots,s_{i,j}-\varepsilon_{i,j}\}$.  
(\ref{eq:11}) with \eqref{eq:dd} implies that 
$$
E_{\alpha}^{(i,j)}E_{\beta}^{(i',j')}=\delta_{\alpha,\beta}\delta_{j,i'}E_{\alpha}^{(i,j')}. 
$$
Therefore $\{E_\ell^{(i,j)} \mid 1\leq i,j\leq n, 0\leq \ell \leq s_{i,j}-\varepsilon_{i,j}\}$ is linearly independent and 
$$
\text{span}\{E_\ell^{(i,j)} \mid 1\leq i,j\leq n, 0\leq \ell\leq s_{i,j}-\varepsilon_{i,j}  \}
$$ 
is closed under ordinary multiplication. 
Since 
$$
\text{span}\{A_\ell^{(i,j)} \mid 1\leq i,j\leq n, \ell\in A(X_i,X_j)  \}=\text{span}\{E_\ell^{(i,j)} \mid 1\leq i,j\leq n, 0\leq \ell\leq s_{i,j}-\varepsilon_{i,j}  \}
$$ 
holds, $(\bigcup_{i=1}^n X_i, \{R_\ell^{(i,j)} \mid 1\leq i,j\leq n,\varepsilon_{i,j}\leq \ell \leq s_{i,j}\}) $ is a coherent configuration.
Krein numbers of the coherent configuration are positive scalar multiple of those for the association scheme, so Proposition~\ref{prop:Q-poly}(3) is satisfied.
\end{proof}

\begin{example}
A $Q$-polynomial association scheme with $d$ classes is $Q$-antipodal if $b_j^*=c_{d-j}^*$ for all $j$ except possibly $j=\lfloor \frac{d}{2}\rfloor$, see \cite{MMW} for more information.  
Let $\mathfrak{X}=(X,\{R_i\}_{i=0}^d)$ be a $Q$-antipodal $Q$-polynomial association scheme with $Q$-antipodal classes $X_1,X_2,\ldots,X_w$.
\cite[Corollary 4.5 and Theorem 4.7]{MMW} imply that 
$s_{i,j}$ is equal to $\lfloor\frac{d}{2} \rfloor$ if $i=j$ and $\lceil \frac{d}{2}\rceil $ otherwise and $t_j=d-1$.
Then $s_{i,j}+s_{j,h}-2\leq t_j$ holds for any $i,j,h\in\{1,2,\ldots,w\}$, 
hence $(\bigcup \nolimits_{i=1}^w X_i, \{R_\ell^{(i,j)} \mid 1\leq i,j\leq \omega,\varepsilon_{i,j}\leq \ell \leq s_{i,j}\}$ is a $Q$-polynomial coherent configuration and this coherent configuration was shown to be uniform \cite[Theorem~5.1]{DMM}. See \cite{DMM} for $Q$-antipodal association schemes and uniform coherent configurations. 
\end{example}

\begin{example}
A $2e$-design $C$ in $H(n,q)$ satisfies an inequality $|C|\leq \sum_{i=0}^e \binom{n}{i}(q-1)^i$. 
A $2e$-design is said to be tight if equality is attained above. 
Tight $4$-designs in $H(n,q)$ have been classified in \cite{N,GSV}. 
The technique developed in \cite{GSV} is to consider derived $t_i=3$-designs $C_1,C_2,\ldots,C_q$ in $H(n-1,q)$ from $C$ such that $s_{i,j}=|A(C_i,C_j)|=2$ for any $i,j\in\{1,2,\ldots,q\}$. 
Since $s_{i,j}+s_{j,h}-1= t_j$ hold for any $i,j,h\in\{1,2,\ldots,q\}$, $\bigcup_{i=1}^q C_i$ with the binary relations forms a $Q$-polynomial coherent configuration.     
\end{example}
\subsection{Spherical designs}
Let $X_1,X_2,\ldots,X_n$ be non-empty finite subsets of the unit sphere $S^{d-1}$ in $\mathbb{R}^d$ such that $X_i\cap X_j=\emptyset$ for $i,j\in\{1,2,\ldots,n\}$ and $X=\bigcup_{i=1}^n X_i$. 
We denote by $\langle x,y\rangle$ the standard inner product of $x, y \in \mathbb{R}^d$.
We define the angle set $A(X_i,X_j)$ between $X_i$ and $X_j$ by
$$A(X_i,X_j)=\{ \langle x, y\rangle \mid x \in X_i , y \in X_j , x \neq y \}.$$
Let $s_{i,j}=|A(X_i,X_j)|$ and 
 $A(X_i,X_j)=\{\alpha_1^{(i,j)},\alpha_2^{(i,j)},\ldots,\alpha_{s_{i,j}}^{(i,j)}\}$ and set $\alpha_0^{(i,i)}=1$.  
Define a subset $R_\ell^{(i,j)}$ of $X\times X$ by 
$$
R_\ell^{(i,j)}=\{(x,y)\in X\times X \mid x\in X_i,y\in X_j, \langle x,y\rangle=\alpha_\ell^{(i,j)} \}.
$$
Denote by $A_\ell^{(i,j)}$ the adjacency matrix of the graph $(X,R_\ell^{(i,j)})$. 

For a positive integer $t$, a non-empty finite set $Y$ in the unit sphere $S^{d-1}$ is called 
a spherical $t$-design in $S^{d-1}$ if the following condition is satisfied:
$$\frac{1}{|Y|}\sum\limits_{y \in Y}f(y)=\frac{1}{|S^{d-1}|}\int\nolimits_{S^{d-1}}f(y)d\sigma(y)$$
for all polynomials $f(x)=f(x_1,\dots,x_d)$ of degree not exceeding $t$.
Here $|S^{d-1}|$ denotes the volume of the sphere $S^{d-1}$.

We define the Gegenbauer polynomials $\{Q_k(x)\}_{k=0}^\infty$ on $S^{d-1}$ by
\begin{align*}
& Q_0(x)=1,\quad Q_1(x)=dx,\\
& \frac{k+1}{d+2k}Q_{k+1}(x)=xQ_k(x)-\frac{d+k-3}{d+2k-4}Q_{k-1}(x).
\end{align*}

Let $\mbox{Harm}(\mathbb{R}^d)$ be the vector space of the harmonic polynomials over $\mathbb{R}$
and $\mbox{Harm}_\ell(\mathbb{R}^d)$ be the subspace of $\mbox{Harm}(\mathbb{R}^d)$ consisting of homogeneous polynomials of total degree $\ell$.
Let $\{\phi_{\ell,1},\dots,\phi_{\ell,h_\ell}\}$ be an orthonormal basis
of $\text{Harm}_\ell(\mathbb{R}^d)$ with respect to the inner product 
$$\langle\phi,\psi \rangle=\frac{1}{|S^{d-1}|}\int\nolimits_{S^{d-1}}\phi(x)\psi(x)d\sigma(x) .$$
Then the addition formula for the Gegenbauer polynomial holds \cite[Theorem 3.3]{DGS}:
\begin{lemma}\label{add}
$\sum\limits_{i=1}^{h_\ell}\phi_{\ell,i}(x)\phi_{\ell,i}(y)=Q_\ell(\langle x,y\rangle)$ for any $\ell\in \mathbb{N}$, $x,y\in S^{d-1}$.
\end{lemma}
We define the $\ell$-th characteristic matrix of a non-empty finite set $Y\subset S^{d-1}$ as the $|Y|\times h_\ell$ matrix
$$ H_\ell=(\phi_{\ell,i}(x))_{\substack{x\in X\\1\leq i\leq h_\ell}} .$$
A criterion for $t$-designs using Gegenbauer polynomials and the characteristic matrices is known \cite[Theorem 5.3, 5.5]{DGS}.
\begin{lemma}\label{cha}
Let $Y$ be a non-empty finite set in $S^{d-1}$. The following conditions are equivalent:
\begin{enumerate}
\item $Y$ is a $t$-design,
\item $\sum\limits_{x,y\in Y}Q_k(\langle x,y\rangle)=0$ for any $k\in \{1,\ldots,t\}$,
\item $H_k^\top H_\ell=\delta_{k,\ell}|Y|I \text{ for } 0\leq k+\ell\leq t$.  
\end{enumerate}
\end{lemma}

For mutually disjoint non-empty finite subsets $X_1,X_2,\ldots,X_n$ of $S^{d-1}$, after suitably rearranging the elements of $X=\bigcup_{i=1}^n X_i$, the $\ell$-th characteristic matrix $H_\ell$ of $X$ has the following form: 
$$
H_\ell=\begin{pmatrix}H_\ell^{(1)}\\ H_\ell^{(2)} \\ \vdots \\ H_\ell^{(n)}\end{pmatrix}=\sum_{i=1}^n e_i\otimes H_\ell^{(i)} 
$$
where $e_i$ denotes the column vector of length $n$ with a $1$ in the $i$-th coordinate and $0$'s elsewhere and  $H_\ell^{(i)}$ is the $\ell$-th characteristic matrix of $X_i$.  
Denote $\tilde{H}_\ell^{(i)}=e_i\otimes H_\ell^{(i)}$. 
For $i\in\{1,2,\ldots,n\}$, define $\Delta_{X_i}$ be the diagonal matrix indexed by the elements of $X$ with $(x,x)$-entry equal to $1$ if $x\in X_i$ and $0$ otherwise.

\begin{theorem}\label{sphere}
Let $X_i$ be a spherical $t_i$-design on $ S^{d-1}$ for $i \in \{1,2,\dots,n\}$.
Assume that $X_i \cap X_j=\emptyset$ for distinct integers $i ,j \in \{1,2,\dots,n\}$.
Let $s_{i,j}=|A(X_i,X_j)|$.
If one of the following holds depending on the choice of $i,j,h \in \{1, 2,\ldots ,n \}$;
\begin{enumerate}
\item $s_{i,j}+s_{j,h}-2 \leq t_j$,
\item $i=j=h$, $2s_{i,i}-3 = t_i$, and $X_i=-X_i$,
\end{enumerate} 
then $(\bigcup_{i=1}^n X_i, \{R_\ell^{(i,j)} \mid 1\leq i,j\leq n, \varepsilon_{i,j}\leq \ell \leq s_{i,j}\})$ is a $Q$-polynomial coherent configuration. 
\end{theorem}
\begin{proof}
In \cite{S}, it is shown that $(\bigcup \nolimits_{i=1}^n X_i, \{R_\ell^{(i,j)} \mid 1\leq i,j\leq n, \varepsilon_{i,j}\leq \ell \leq s_{i,j}\})$ is a coherent configuration.

For any $i,j\in\{1,2,\ldots,n\}$ and $\ell \in\{0,1,\ldots,s_{i,j}+\varepsilon_{i,j}\}$, we define
\begin{align*}
c_\ell^{(i,j)}&=\left\{\begin{array}{ll}
\frac{(s_{i,i}-2)!(d-1)!|X_i|-2(s_{i,i}-2)(d+s_{i,i}-4)!}{2d(d-1)(d+s_{i,i}-4)!} & \text{if}\ i=j \text{ and}\ t_i=2s_{i,i}-3, \ell=s_{i,i}-2, \\
 1  & \text{otherwise},
 \end{array}\right.  \\
E_\ell^{(i,j)}&=\left\{\begin{array}{ll}
\frac{c_\ell^{(i,j)}}{\sqrt{|X_i||X_j|}}\tilde{H}_\ell^{(i)}(\tilde{H}_\ell^{(j)})^\top & \text{ if } \ell\leq s_{i,j}-1,\\
\Delta_{X_i}-\frac{1}{|X_i|}\sum\limits_{m=0}^{s_{i,i}-1}c_m^{(i,j)}\tilde{H}_m^{(i)}(\tilde{H}_m^{(j)})^\top & \text{ if }\ i=j \text{ and } \ell=s_{i,i}.
 \end{array}\right.
\end{align*}
Thus (B1) and (B3) hold.  
Note that $E_{s_{i,i}}^{(i,i)}=\Delta_{X_i}-\sum\limits_{m=0}^{s_{i,i}-1}E_m^{(i,i)}$.

For $x\in X_i$, $y\in X_j$,
\begin{align*}
E_\ell^{(i,j)}(x,y)&=\left\{\begin{array}{ll}
\frac{c_\ell^{(i,j)}}{\sqrt{|X_i||X_j|}}\phi_\ell(x){\phi_\ell(y)}^\top & \text{ if } \ell\leq s_{i,j}-1,\\
\delta_{x,y}-\frac{1}{|X_i|}\sum\limits_{m=0}^{s_{i,i}-1}c_m^{(i,j)}\phi_m(x){\phi_m(y)}^\top & \text{ if } i=j\ \text{ and } \ell=s_{i,j},
 \end{array}\right. \\
 &=\left\{\begin{array}{ll}
\frac{c_\ell^{(i,j)}}{\sqrt{|X_i||X_j|}}Q_\ell(\langle x,y\rangle) & \text{ if } \ell\leq s_{i,j}-1,\\
\delta_{x,y}-\frac{1}{|X_i|}\sum\limits_{m=0}^{s_{i,i}-1}c_m^{(i,j)}Q_m(\langle x,y\rangle) & \text{ if } i=j \text{ and } \ell=s_{i,i}.
 \end{array}\right. 
\end{align*}
Therefore 
\begin{align}
E_\ell^{(i,j)}=\left\{\begin{array}{ll}
\frac{c_\ell^{(i,j)}}{\sqrt{|X_i||X_j|}}\sum\limits_{k=\varepsilon_{i,j}}^{s_{i,j}}Q_\ell(\alpha_{i,j}^k)A_k^{(i,j)} & \text{ if } \ell\leq s_{i,i}-1,\\
\Delta_{X_i}-\frac{1}{|X_i|}\sum\limits_{k=\varepsilon_{i,j}}^{s_{i,i}}\biggl(\sum\limits_{m=0}^{s_{i,i}-1}c_m^{(i,i)}Q_m(\alpha_{i,i}^k)\biggr)A_k^{(i,i)} & \text{ if } i=j \text{ and } \ell=s_{i,i}.
\end{array}\right. \label{eq1}
\end{align}
This implies that $E_\ell^{(i,j)} \in \mathcal{A}^{(i,j)}$. 

We show that $E_{\alpha}^{(i,j)}E_{\beta}^{(i',j')}=\delta_{\alpha,\beta}\delta_{j,i'}E_{\alpha}^{(i,j')}$. 
If $j\neq i'$, then by $e_j^\top e_{i'}=0$,
\begin{align*}
E_{\alpha}^{(i,j)}E_{\beta}^{(i',j')}&=\frac{1}{\sqrt{|X_i||X_j||X_{i'}||X_{j'}|}}\tilde{H}_{\alpha}^{(i)}(\tilde{H}_{\alpha}^{(j)})^\top \tilde{H}_{\beta}^{(i')}(\tilde{H}_{\beta}^{(j')})^\top\displaybreak[0]\\
&=\frac{1}{\sqrt{|X_i||X_j||X_{i'}||X_{j'}|}}\tilde{H}_{\alpha}^{(i)}\left((e_j\otimes H_{\alpha}^{(j)})^\top (e_{i'}\otimes H_{\beta}^{(i')})\right)(\tilde{H}_{\beta}^{(j')})^\top\displaybreak[0]\\
&=\frac{1}{\sqrt{|X_i||X_j||X_{i'}||X_{j'}|}}\tilde{H}_{\alpha}^{(i)}\left(e_j^\top e_{i'}\otimes (H_{\alpha}^{(j)})^\top H_{\beta}^{(i')}\right)(\tilde{H}_{\beta}^{(j')})^\top\displaybreak[0]\\
&=0.
\end{align*} 
In the following, we assume $j=i'$ and then show that $E_{\alpha}^{(i,j)}E_{\beta}^{(j,h)}=\delta_{\alpha,\beta}E_{\alpha}^{(i,h)}$. 

\begin{enumerate}
\item The case where $i,j,h$ satisfying $s_{i,j}+s_{j,h}-2\leq t_j$.  
In order to prove $E_{\alpha}^{(i,j)}E_{\beta}^{(j,h)}=\delta_{\alpha,\beta}E_{\alpha}^{(i,h)}$ for $\alpha\in\{0,1,\ldots,s_{i,j}-\varepsilon_{i,j}\},\beta\in\{0,1,\ldots,s_{j,h}-\varepsilon_{j,h}\}$,  the case is divided into the following cases. 
\begin{enumerate}
\item For $\alpha,\beta$ satisfying $\alpha\leq s_{i,j}-1$, $\beta\leq s_{j,h}-1$, 
\begin{align*}
E_{\alpha}^{(i,j)}E_{\beta}^{(j,h)}&=\frac{1}{\sqrt{|X_i||X_h|}|X_j|}\tilde{H}_{\alpha}^{(i)}(\tilde{H}_{\alpha}^{(j)})^\top \tilde{H}_{\beta}^{(j)}(\tilde{H}_{\beta}^{(h)})^\top\displaybreak[0]\\
&=\frac{\delta_{\alpha,\beta}}{\sqrt{|X_i||X_h|}}\tilde{H}_{\alpha}^{(i)}(\tilde{H}_{\alpha}^{(h)})^\top\displaybreak[0]\\
&=\delta_{\alpha,\beta}E_{\alpha}^{(i,h)}.
\end{align*}

\item When $i=j$, for $\alpha,\beta$ satisfying $\alpha=s_{i,i}$, $\beta\leq s_{i,h}-1$, $s_{i,h}\leq s_{i,i}+1$ holds by \cite[p.227]{Hig1}. Then $s_{i,i}-1< \beta$ hold if and only if $\beta=s_{i,h}-1=s_{i,i}$, and   
\begin{align*}
E_{s_{i,i}}^{(i,i)}E_{\beta}^{(i,h)}&=\left(\Delta_{X_i}-\sum\limits_{m=0}^{s_{i,i}-1}E_m^{(i,i)}\right) E_{\beta}^{(i,h)} \displaybreak[0]\\
&=E_{\beta}^{(i,h)}-\sum\limits_{m=0}^{s_{i,i}-1}E_m^{(i,i)}E_{\beta}^{(i,h)} \displaybreak[0]\\
&=\begin{cases}E_{\beta}^{(i,h)}-E_{\beta}^{(i,h)} & \text{ if } \beta\leq s_{i,i}-1\\
E_{\beta}^{(i,h)} & \text{ if } \beta=s_{i,h}-1= s_{i,i}
\end{cases}\\
&=\delta_{s_{i,i},\beta}E_{s_{i,i}}^{(i,h)}. 
\end{align*}
A similar is true for $j=k$ and $\alpha\leq s_{i,j}-1,\beta=s_{j,j}$. 

\item When $i=j=h$, for $\alpha=\beta= s_{i,i}$, 
\begin{align*}
E_{s_{i,i}}^{(i,i)}E_{s_{i,i}}^{(i,i)}&=\left(\Delta_{X_i}-\sum\limits_{\ell=0}^{s_{i,i}-1}E_\ell^{(i,i)}\right) \left(\Delta_{X_i}-\sum\limits_{m=0}^{s_{i,i}-1}E_m^{(i,i)}\right) \displaybreak[0]\\
&=\Delta_{X_i}-2\sum\limits_{\ell=0}^{s_{i,i}-1}E_m^{(i,i)}+\sum\limits_{\ell,m=0}^{s_{i,i}-1}E_\ell^{(i,i)}E_m^{(i,i)} \displaybreak[0]\\
&=\Delta_{X_i}-2\sum\limits_{\ell=0}^{s_{i,i}-1}E_m^{(i,i)}+\sum\limits_{\ell=0}^{s_{i,i}-1}E_\ell^{(i,i)} \displaybreak[0]\\
&=E_{\beta}^{(i,k)}-E_{\beta}^{(i,k)}\\
&=E_{s_{i,i}}^{(i,i)}. 
\end{align*}
\end{enumerate}
\item The case where $i,j,h$ satisfying $i=j=h$, $t_j=s_{i,i}+s_{i,i}-3$ and $X_i=-X_i$.  
It was shown in \cite{BB} that $E_{\alpha}^{(i,i)}E_{\beta}^{(i,i)}=\delta_{\alpha,\beta}E_{\alpha}^{(i,i)}$ for $\alpha,\beta\in\{0,1,\ldots,s_{i,i}\}$.
\end{enumerate} 
This proves that (B4) holds. 

Next we show (B2). Fix $i,j\in\{1,2,\ldots,n\}$.  
For $d_{\ell}^{(i,j)}\in \mathbb{C}$, let $\sum\limits_{i,j=1}^n\sum\limits_{\ell=0}^{s_{i,j}-\varepsilon_{i,j}} d_{\ell}^{(i,j)}E_{\ell}^{(i,j)}=0$.
Multiplying $E_m^{(i,i)}$ on the left side and $E_m^{(j,j)}$ on the right side, 
we obtain $d_{m}^{(i,j)}E_{m}^{(i,j)}=0$ and  thus $d_{m}^{(i,j)}=0$ for any $m$. 
Therefore $\{E_\ell^{(i,j)} \mid 0\leq \ell \leq s_{i,j}-\varepsilon_{i,j}\}$ is a linear independent set over $\mathbb{C}$.
Since $\dim\mathcal{A}^{(i,j)}=s_{i,j}+\delta_{i,j}$, $\{E_\ell^{(i,j)} \mid 0\leq \ell \leq s_{i,j}-\varepsilon_{i,j}\}$ is a basis of $\mathcal{A}^{(i,j)}$.  
This proves (B2).

It remains to prove $Q$-polynomiality.
Setting 
\begin{align*}
v_\ell^{(i,j)}(x)=\left\{\begin{array}{ll}c_\ell^{(i,j)}Q_\ell(\frac{x}{d}) & \text{ if } \ell\leq s_{i,i}-1,\\
|X_i|F_{A(X_i)}(\frac{x}{d})-\sum\limits_{m=0}^{s_{i,i}-1}c_m^{(i,i)}Q_m(\frac{x}{d}) & \text{ if } i=j\ \text{ and } \ell=s_{i,i},
\end{array}\right. 
\end{align*}
where $F_{A(X_i)}(x)=\prod_{\alpha\in A(X_i)}\frac{x-\alpha}{1-\alpha}$ and $A(X_i)=A(X_i,X_i)$. 
Equation \eqref{eq1} implies that for $i,j\in\{1,2,\ldots,n\},\ell\in \{0,1,\ldots, s_{i,j}-\varepsilon_{i,j}\}$,
$$\sqrt{|X_i||X_j|}E_\ell^{(i,j)}=v_\ell^{(i,j)}(\sqrt{|X_i||X_j|}E_1^{(i,j)})$$
under the entry-wise product, which implies that Proposition~\ref{prop:Q-poly}(1) is satisfied. 
This completes the proof.   
\end{proof}
\begin{example}
Let $X$ be a non-empty finite set in $S^{d-1}$ with the angle set $A(X)=\{\langle x,y\rangle \mid x,y\in X,x\neq y\}=\{\alpha_1,\alpha_2,\ldots,\alpha_s\}$ where $\alpha_1>\cdots>\alpha_s$.   
After suitably transforming the set $X$, we may assume that $e_1\in X$.   
For $i\in\{1,2,\ldots,s\}$ such that $\alpha_i\neq -1$, the derived code $X_i$ with respect to $e_1$ is defined to be 
$$
X_i=\{x \in S^{d-2} \mid (\alpha_i, \sqrt{1-\alpha_i^2}x)\in X\}.
$$ 
Suppose that $X$ be a spherical $t$-design in $S^{d-1}$ and let $s^*=|A(X)\setminus \{-1\}|$. 
For $i,j\in\{1,\ldots,s^*\}$, the angle set between $X_i$ and $X_j$ satisfies 
$$
A(X_i,X_j)\subset\left\{\tfrac{\alpha_h-\alpha_i \alpha_j}{\sqrt{(1-\alpha_i^2)(1-\alpha_j^2)}} \mid 1\leq h\leq s\right\}.
$$
Therefore $s_{i,j}~=|A(X_i,X_j)|$ satisfies that $s_{i,j}\leq s$. 

It is shown in \cite[Theorem~8.2]{DGS} that if $t+1\geq s^*$, then $X_i$ is a spherical $(t+1-s^*)$-design in $S^{d-2}$. 
This design is said to be the derived design.

Let $X\subset S^{d-1}$ be a spherical tight $4$-, $5$-, $7$-design with $s^*=|A(X)\setminus\{-1\}|$ and
$X_i$ be a derived design in $S^{d-1}$ for $i\in\{1,\ldots,s^*\}$. 
Since $s_{i,j}+s_{j,h}-2\leq t_j$ holds for $i,j,h\in\{1,\ldots,s^*\}$ in each case,
Theorem~\ref{sphere} implies that  $(\bigcup \nolimits_{i=1}^{s^*} X_i, \{R_\ell^{(i,j)} \mid 1\leq i,j\leq n, \varepsilon_{i,j}\leq \ell\leq s_{i,j}\})$ is a $Q$-polynomial coherent configuration. 

The second eigenmatrices of the $Q$-polynomial coherent configuration obtained from tight $4$-design are given as follows:
\begin{align*}
Q^{(1,1)}&=\begin{pmatrix}
1&d-1& \frac{(d-2)(d+\sqrt{d+3}+1)}{4}\\
1&\sqrt{d+3}-2&-\sqrt{d+3}+1\\
1&-\frac{d+\sqrt{d+3}-3}{\sqrt{d+3}-1}&\frac{d-2}{\sqrt{d+3}-1}
\end{pmatrix},
Q^{(1,2)}=Q^{(2,1)}=
\begin{pmatrix}
1&\sqrt{d-1} \\
1&-\sqrt{d-1}
\end{pmatrix},\\
Q^{(2,2)}&=\begin{pmatrix}
1&d-1&\frac{(d-2)(d-\sqrt{d+3}+1)}{4} \\
1&\frac{d-\sqrt{d+3}-3}{\sqrt{d+3}+1}&\frac{-d+2}{\sqrt{d+3}+1}\\
1&-\sqrt{d+3}-2&\sqrt{d+3}+1
\end{pmatrix}.
\end{align*}

The second eigenmatrices of the $Q$-polynomial coherent configuration obtained from tight $5$-design are given as follows:
\begin{align*}
Q^{(i,j)}&=\begin{pmatrix}
1&d-1&\frac{1}{2}(d-2)(d+1)\\
1&\frac{(d-1)(\sqrt{d+2}-1)}{d+1}&-\frac{(d-1)\sqrt{d+2}+2}{d+1}\\
1&-\frac{(d-1)(\sqrt{d+2}+1)}{d+1}&\frac{(d-1)\sqrt{d+2}-2}{d+1}
\end{pmatrix} \text{ for } i=j\in\{1,2\},\\
Q^{(i,j)}&=\begin{pmatrix}
1&-\frac{(d-1)(\sqrt{d+2}+1)}{d+1}&\frac{(d-1)\sqrt{d+2}-2}{d+1}\\
1&\frac{(d-1)(\sqrt{d+2}-1)}{d+1}&-\frac{(d-1)\sqrt{d+2}+2}{d+1}\\
1&-d+1&\frac{1}{2}(d-2)(d+1)
\end{pmatrix} \text{ for } i\neq j\in\{1,2\}.
\end{align*}

The second eigenmatrices of the $Q$-polynomial coherent configuration obtained from tight $7$-design are given as follows:
\begin{align*}
Q^{(i,j)}&=\begin{pmatrix}
1&d-1&\frac{(d-2)(d+1)}{2}&\frac{(d-2)(d-1)(d+1)}{18}\\
1&\frac{(d-1)(\sqrt{3(d+4)}-3)}{d+1}&\frac{(\sqrt{3(d+4)}-3)(d-5-\sqrt{3(d+4)})}{\sqrt{3(d+4)}+3}&-\frac{(d-1)(\sqrt{3(d+4)}-3)^2}{3(d+1)}\\
1&\frac{3(d-1)}{d+1}&-\frac{(d-5)(d-2)}{2(d+1)}&\frac{(d-2)(d-1)}{2(d+1)}\\
1&-\frac{(d-1)(\sqrt{3(d+4)}+3)}{d+1}&\frac{(\sqrt{3(d+4)}-3)(d-5-\sqrt{3(d+4)})}{\sqrt{3(d+4)}+3}&-\frac{(d-1)(\sqrt{3(d+4)}+3)^2}{3(d+1)}
\end{pmatrix} \text{ for } i=j\in \{1,3\},\\
Q^{(i,j)}&=\begin{pmatrix}
1&\frac{(d-1)(\sqrt{3(d+4)}+3)}{d+1}&\frac{3(\sqrt{3(d+4)}+3)(d-5+\sqrt{3(d+4)})}{d+1}&-\frac{(d-1)(\sqrt{3(d+4)}+3)^2}{3(d+1)}\\
1&-\frac{3(d-1)}{d+1}&-\frac{(d-5)(d-2)}{2(d+1)}&\frac{(d-2)(d-1)}{2(d+1)}\\
1&-\frac{(d-1)(\sqrt{3(d+4)}-3)}{d+1}&\frac{3(\sqrt{3(d+4)}-3)(d-5-\sqrt{3(d+4)})}{d+1}&-\frac{(d-1)(\sqrt{3(d+4)}-3)^2}{3(d+1)}\\
1&-d+1&\frac{(d-2)(d+1)}{2}&\frac{(d-2)(d-1)(d+1)}{18}
\end{pmatrix} \text{ for } i\neq j\in \{1,3\},\displaybreak[0]\\
Q^{(2,2)}&=\begin{pmatrix}
1&d-1&\frac{(d-2)(d+1)}{2}&\frac{(d-2)(d-1)(d+4)}{9}&\frac{(d-2)(d-1)(2d-1)}{18}\\
1&\frac{\sqrt{3}(d-1)}{\sqrt{d+4}}&\frac{(2d-7)(d+1)}{2(d+4)}&-\frac{\sqrt{3}(d-1)}{\sqrt{d+4}}&-\frac{(2d-1)(d-1)}{2(d+4)}\\
1&0&-\frac{1}{2}(d+1)&0&\frac{1}{2}(d-1)\\
1&-\frac{\sqrt{3}(d-1)}{\sqrt{d+4}}&\frac{(2d-7)(d+1)}{2(d+4)}&\frac{\sqrt{3}(d-1)}{\sqrt{d+4}}&-\frac{(2d-1)(d-1)}{2(d+4)}\\
1&-d+1&\frac{(d-2)(d+1)}{2}&-\frac{(d-2)(d-1)(d+4)}{9}&\frac{(d-2)(d-1)(2d-1)}{18}
\end{pmatrix},\displaybreak[0]\\
Q^{(i,j)}&=\begin{pmatrix}
1&\frac{\sqrt{3}(d-1)}{\sqrt{d+1}}&d-2\\
1&0&-\frac{d+1}{2}\\
1&-\frac{\sqrt{3}(d-1)}{\sqrt{d+1}}&d-2
\end{pmatrix} \text{ for } (i,j)\in\{(1,2),(3,2),(2,1),(2,3)\}.
\end{align*}
\end{example}

Let $M=\{M_i\}_{i=1}^f$ be a collection of orthonormal bases of $\mathbb{R}^d$.
The set $M$ is called  real mutually unbiased bases (MUB) if any two vectors $x$ and $y$ from different bases
satisfy $\langle x,y \rangle =\pm 1/\sqrt{d}$.
It is known that the number $f$ of real mutually unbiased bases in $\mathbb{R}^d$ can be at most $d/2+1$.  
We call $M$ a maximal MUB if this upper bound is attained. 

The assumption of Theorem~\ref{sphere} is not satisfied for a union of derived codes of maximal MUB, 
but the same conclusion holds.
\begin{theorem}
Let $M_1,M_2,\ldots,M_{d/2+1}$ be a maximal MUB of $\mathbb{R}^d$, 
$X=\bigcup_{i=1}^{d/2+1}(X_i\cup(-X_i))$
and $X_i$ be the derived design of $X$ relative to a point in $X$ for $i\in\{1,2,3\}$.
Then $\mathcal{C}=(\bigcup\nolimits_{i=1}^3 X_i, \{R_\ell^{(i,j)} \mid 1\leq i,j\leq n, \varepsilon_{i,j}\leq \ell\leq s_{i,j}\})$ is a $Q$-polynomial coherent configuration.
\begin{proof}
In \cite{S}, it is shown that $\mathcal{C}$ is a coherent configuration. 
In what follows, we construct a basis $E_\ell^{(i,j)}$ which has  a $Q$-polynomial property. 
 
The type of $\mathcal{C}$ is 
$$(s_{i,j}+\delta_{i,j})_{i,j=1}^3=\left(\begin{array}{ccc} 4&2&4\\2&3&2\\4&2&4\end{array}\right).$$
We define $E_\ell^{(i,j)}$ for $i,j\in\{1,2,3\}, 0\leq \ell \leq s_{i,j}-\varepsilon_{i,j}$ as follows.
\begin{itemize}
\item For $i,j\in\{1,2,3\}$, $\ell\in\{0,1\}$,  $E_\ell^{(i,j)}=\frac{1}{\sqrt{|X_i||X_j|}}\tilde{H}_\ell^{(i)}(\tilde{H}_\ell^{(j)})^\top$.
\item For $i,j\in\{1,3\}$, $E_2^{(i,j)}=\frac{d+1}{(d-1)\sqrt{|X_i||X_j|}}\tilde{H}_2^{(i)}(\tilde{H}_2^{(j)})^\top$. 
\item For $i\in\{1,2,3\}$, $E_{s_{i,i}}^{(i,i)}=\Delta_{|X_i|}-\sum\limits_{k=0}^{s_{i,i}-1}E_k^{(i,i)}$.
\item For $\{i,j\}=\{1,3\}$, $E_{3}^{(i,j)}=A_4^{(i,j)}-\sum\limits_{k=0}^2E_k^{(i,j)}$, where $A_4^{(i,j)}$ be the adjacency matrix defined by inner product $-1$ between $X_i$ and $X_j$.
\end{itemize}

It is clear that for $i,j,i',j'\in\{1,2,3\}$ with $j\neq i'$, $\ell\in\{0,1,\ldots,s_{i,j}-\varepsilon_{i,j}\},m\in\{0,1,\ldots,s_{i',j'}-\varepsilon_{i',j'}\}$, $E_\ell^{(i,j)}E_m^{(i',j')}=0$. 
Therefore we will show that for $i,j,h\in\{1,2,3\}$ and $\ell\in\{0,1,\ldots,s_{i,j}-\varepsilon_{i,j}\},m\in\{0,1,\ldots,s_{j,h}-\varepsilon_{j,h}\}$, $E_\ell^{(i,j)}E_m^{(j,h)}=\delta_{\ell,m}E_\ell^{(i,h)}$. 

It is shown in \cite{ABS} that $X_1$ and $X_3$ are $Q$-polynomial association schemes, and they are isomorphic.
The polynomial $v_2(x)$ of degree $2$ which is determined from $Q$-polynomiality of $X_1$ and $X_3$ is $\frac{d+1}{d-1}Q_2(\frac{x}{d-1})$, 
so $E_2^{(i,i)}=\frac{d+1}{(d-1)|X_i|}\tilde{H}_2^{(i)}(\tilde{H}_2^{(i)})^\top$ is a primitive idempotent for $i=1,3$.
Then we have $A_4^{(i,j)}\tilde{H}_k^{(j)}=\tilde{H}_k^{(i)}$.
Therefore $E_\ell^{(i,i)}E_m^{(i,i)}=\delta_{\ell,m}E_\ell^{(i,i)}$ for $i\in\{1,3\}$ if and only if $E_\ell^{(i,j)}E_m^{(j,h)}=\delta_{\ell,m}E_\ell^{(i,h)}$ for $i,j,h\in\{1,3\}$. 
This completes the proof.  
\end{proof}
\end{theorem}
The second eigenmatrices of $Q$-polynomial coherent configuration obtained from MUB are given as follows:
\begin{align*}
Q^{(i,j)}&=\begin{pmatrix}
1&d-1&\frac{1}{2}(d-2)(d-1)&\frac{d}{2}-1\\
1&\sqrt{d}+1&\sqrt{d}-1&-1\\
1&-1&-\frac{d}{2}+1&\frac{d}{2}-1\\
1&-\sqrt{d}+1&-\sqrt{d}-1&-1
\end{pmatrix}\text{ for }i,j\in \{1,3\},\\
Q^{(2,2)}&=\begin{pmatrix}
1&d-1&d-2\\
1&0&-1\\
1&-d+1&d-2
\end{pmatrix},\\
Q^{(i,j)}&=\begin{pmatrix}
1&\sqrt{d-1}\\
1&-\sqrt{d-1}
\end{pmatrix} \text{ for } (i,j)\in\{(1,2),(3,2),(2,1),(2,3)\}.
\end{align*}

\section{The Terwillger algebra of $H(n,2)$ and tight relative $2e$-designs on two shells}\label{sec:hn2}
\subsection{The Terwillger algebra of $H(n,2)$}
The Terwilliger algebra \cite{T1} of the binary Hamming schemes $H(n,2)=(X,\{R_i\}_{i=0}^n)$ is a coherent configuration because the scheme $H(n,2)$ is triply regular, that is, for $(x,y)\in R_i,(y,z)\in R_j,(z,x)\in R_h$, the number $|\{w\in X \mid (x,w),\in R_{j'},(y,w),\in R_{h'},(z,w),\in R_{i'}\}|$ depends only on $i,j,h,i',j',h'$, not on the choice of $x,y,z$. 
We include the result by Vallentin \cite{V} for the basis of the coherent configuration. See also \cite{Sc}.  

Let  $n$ be a positive integer, $X=\{0,1\}^n$ and $R_i=\{(x,y)\in X\times X \mid d(x,y)=i\}$ for $i\in\{0,1,\ldots,n\}$ where $d(x,y)=|\{\ell\in\{1,\ldots,n\} \mid x_\ell\neq y_\ell \}|$ is the Hamming distance.    
The binary Hamming scheme is a pair $(X,\{R_i\}_{i=0}^n)$. 
For $i\in\{0,1,\ldots,n\}$, define $X_i=\{x\in X\mid d(x,0)=i\}$ where $0=(0,\ldots,0)$.

Define $(a)_0=1,(a)_k=a(a+1)\cdots(a+k-1)$ and 
$$
Q_k(x;-a-1,-b-1,m)=\frac{1}{\binom{m}{k}}\sum_{j=0}^k (-1)^j \frac{\binom{b-k+j}{j}}{\binom{a}{j}}\binom{m-x}{k-j}\binom{x}{j} 
$$ 
to be Hahn polynomials of degree with respect to $x$ (for integers $m,a,b$ with $a\geq m,b\geq m\geq 0$).   
\begin{theorem}[See {\cite[Theorem~4.1]{V}}]
For $x,y \in X$, define $v(x,y)=|\{\ell\in\{1,\ldots,n\} \mid x_\ell=1,y_\ell=0 \}|$. 
For $k\in\{0,\ldots,\lfloor n/2\rfloor\}$ and $i,j\in\{k,\ldots,n-k\}$,
define 
$$
E_{k,i,j}(x,y)=\begin{cases}\frac{\binom{n}{k}-\binom{n}{k-1}}{(\binom{n}{i}\binom{n}{j})^{1/2}}\left(\frac{(-j)_k(i-n)_k}{(-i)_k(j-n)_k}\right)^{-\frac{1}{2}}Q_k(v(x,y);-(n-i)-1,-i-1,j) & \text{ if }x\in X_i,y\in X_j,\\
0 & \text{ if }x\not\in X_i \text{ or } y\not\in X_j. 
\end{cases}
$$
Then $E_{k,i,j}$ ($k\in\{0,\ldots,\lfloor n/2\rfloor\}$ and $i,j\in\{k,\ldots,n-k\}$) form a basis satisfying (B1)-(B4). 
In particular, the Terwilliger algebra of $H(n,2)$ is a $Q$-polynomial coherent configuration.  
\end{theorem}

\subsection{Tight relative $2e$-designs in $H(n,2)$ on two shells}
It was shown in  \cite[Theorem~5.3]{BBTY2020} that a tight relative $t$-design in $H(n,2)=(X,\{R_i\}_{i=0}^n)$ on two shells yields a coherent configuration. 
We include the result and claim that the resulting coherent configuration is $Q$-polynomial. 

 A weighted subset of $X$ a pair $(Y, \omega)$ of a subset $Y$ of $X$ and a function $\omega:Y \rightarrow (0, \infty)$. 
Define the characteristic vector $\chi=\chi_{Y,\omega}$ of a weighted subset $(Y,\omega)$ by $\chi(x)$ equals to $\omega(x)$ if $x\in Y$ and $0$ if $x\not\in Y$. 
A weighted subset $(Y, \omega)$ is said to be  a relative $t$-design with respect to $x\in X$  if $E_i\chi\in\text{span}\{E_i \hat{x}\}_{i=1}^t$ where $\hat{x}$ is the characteristic vector of $\{x\}$. 
Here we assume $x=(0,\ldots,0)$ and set $L=L_Y=\{\ell \mid Y\cap X_\ell\not=\emptyset\}$. 
Then we say that $(Y,\omega)$ is supported on $\cup_{\ell \in L}X_\ell$. 
The following is Fisher type inequality due to \cite{BB2012} and \cite{X}: for a realtive 2e-design supported on $\bigcup_{\ell \in L}X_\ell$, $|Y|\geq \sum_{i=0}^{\min\{|L|-1,e\}}\binom{n}{e-i}$. 
A relative $2e$-design is tight if equality holds above. 

Let $(Y,\omega)$ be a tight relative $2e$-design on two shells $X_\ell\cup X_m$ where $e\leq \ell \leq m \leq n-\ell$, that is $|Y|=\binom{n}{e}+\binom{n}{e-1}$.  
Then  $Y_i=Y \cap X_i$ is a $t_i:=(2e-1)$-design in $J(n,i)$ for $i\in\{\ell,m\}$ and the degree $s_{i,j}$ between $Y_i$ and $Y_j$ is at most $e$. 
It was shown in \cite[Theorem~5.3]{BBTY2020} that $Y_\ell\cup Y_m$ yields a $Q$-polynomial coherent configuration. 

Inspired by this theorem, we show the following theorem, which generalizes Theorem~\ref{thm:Q-polyccDel} to $Q$-polynomial coherent configurations.  
We use the following notation. 
 For a $Q$-polynomial coherent configuration $(\bigcup_{i=1}^nX_i,\{R_\ell^{(i,j)} \mid 1\leq i,j\leq n,\varepsilon_{i,j}\leq \ell \leq r_{i,j}\})$ with fibers $X_1,X_2,\ldots,X_n$ and a subset $Y_i$ of $X_i$ for $i\in \{1,2,\ldots,n\}$, define 
$$
A(Y_i,Y_j)=\{\ell \mid 1\leq \ell \leq r_{i,j},R_\ell^{(i,j)} \cap(Y_i\times Y_j)\neq \emptyset\},
$$
and set $s_{i,j}=|A(Y_i,Y_j)|$. 
For $\ell\in A(Y_i,Y_j)$, define 
$$
\tilde{R}_\ell^{(i,j)}=\{(x,y)\in \bigcup_{i=1}^nX_i\times \bigcup_{i=1}^nX_i \mid x\in Y_i,y\in Y_j, (x,y)\in R_\ell^{(i,j)} \}.
$$
For $i\in\{1,2,\ldots,n\}$, define $\Delta_{Y_i}$ be the diagonal matrix indexed by the elements of $\bigcup_{i=1}^nX_i$ with $(x,x)$-entry equal to $1$ if $x\in Y_i$ and $0$ otherwise, and 
$\tilde{\Delta}_{Y_i}$ as the matrix obtained from $\Delta_{Y_i}$ by restricting the rows to $\bigcup_{i=1}^n Y_i$. 
Note that 
\begin{align}\label{eq:dd1}
(\tilde{\Delta}_{Y_i})^\top\tilde{\Delta}_{Y_i}=\Delta_{Y_i}\text{ and }\tilde{\Delta}_{Y_i}\Delta_{Y_i}=\tilde{\Delta}_{Y_i}.
\end{align}

\begin{theorem}\label{thm:Q-polyccDel2}
Let $\mathcal{C}$ be a $Q$-polynomial coherent configuration with fibers $X_1,X_2,\ldots,X_n$.
Let $Y_i$ be a $t_i$-design in a $Q$-polynomial scheme on $X_i$ with $d_i=r_{i,i}$ classes for $i \in \{1,2,\dots,n\}$.
Define $s_{i,j}=|\{\ell\in\{1,2,\ldots,r_{i,j}\} \mid R_{\ell}^{(i,j)}\cap(Y_i\times Y_j)\neq \emptyset\}|$. 
If $s_{i,j}+s_{j,h}-2 \leq t_j$ holds for any $i,j,h \in \{1,2,\dots ,n\}$, 
then $(\bigcup_{i=1}^n Y_i, \{\tilde{R}_\ell^{(i,j)} \mid 1\leq i,j\leq n,\varepsilon_{i,j}\leq \ell \leq s_{i,j}\})$ is a $Q$-polynomial coherent configuration.
\end{theorem}
\begin{proof}
Let $i,j,h\in \{1,2,\ldots,n\}$, $\alpha\in\{0,1,\ldots,s_{i,j}-1\}$, $\beta\in\{0,1,\ldots, s_{j,h}-1\}$. 
Since $s_{i,j}+s_{j,h}-2\leq t_j$, it holds that $\alpha+\beta+1\leq t_j+1$ and $q_{\alpha,\beta}^\ell=0$ for $\ell\geq t_j+1$. 
Since $Y_j$ is a $t_j$-design, the dual distribution $(b_\ell^{(j)})_{\ell=0}^{d_{j}}$ of $Y_j$ satisfies that $b_\ell^{(j)}=0$ for $\ell\leq t_j$. Then  
\begin{align*}
|||X_j|E_{\alpha}^{(i,j)}\Delta_{Y_j} E_{\beta}^{(j,h)}-|Y_j|\delta_{\alpha,\beta}E_{\alpha}^{(i,h)}||^2&=|Y_j|\sum\limits_{\ell=1}^{d_j} q_{\alpha,\beta,\ell}^{(j,j)} b_\ell^{(j)} \displaybreak[0]\\
&=|Y_j|\left(\sum\limits_{\ell=1}^{t_j} q_{\alpha,\beta,\ell}^{(j,j)} b_\ell^{(j)}+\sum\limits_{\ell=t_j+1}^{d_j} q_{\alpha,\beta,\ell}^{(j,j)} b_\ell^{(j)}\right) \\
&=0.
\end{align*}
Therefore 
$$
|X_j|E_{\alpha}^{(i,j)}\Delta_{Y_j} E_{\beta}^{(j,h)}=|Y_j|\delta_{\alpha,\beta}E_{\alpha}^{(i,h)}. 
$$
Multiplying $\Delta_{Y_i}$ on the left side and $\Delta_{Y_h}$ on the right side, we obtain
\begin{align}
|X_j|\Delta_{Y_i}E_{\alpha}^{(i,j)}\Delta_{X_j} E_{\beta}^{(j,h)}\Delta_{Y_h}=|Y_j|\delta_{\alpha,\beta}\Delta_{Y_i}E_{\alpha}^{(i,h)}\Delta_{Y_h}. \label{eq:qc11}
\end{align}
Define 
\begin{align*}
\tilde{A}_\ell^{(i,j)}=\tilde{\Delta}_{Y_i}A_\ell^{(i,j)}(\tilde{\Delta}_{Y_j})^\top,\quad 
\tilde{E}_{\ell'}^{(i,j)}=\frac{\sqrt{|X_i||X_j|}}{\sqrt{|Y_i||Y_j|}}\tilde{\Delta}_{Y_i}E_{\ell'}^{(i,j)}(\tilde{\Delta}_{Y_j})^\top 
\end{align*}
for $i,j\in\{1,2,\ldots,n\}, \ell \in A(Y_i,Y_j), \ell' \in \{0,1,\ldots,s_{i,j}-\varepsilon_{i,j}\}$.  
\eqref{eq:qc11} with \eqref{eq:dd1} implies that 
$$
\tilde{E}_{\alpha}^{(i,j)}\tilde{E}_{\beta}^{(i',j')}=\delta_{\alpha,\beta}\delta_{j,i'}\tilde{E}_{\alpha}^{(i,j')}. 
$$
Therefore $\{\tilde{E}_\ell^{(i,j)} \mid 1\leq i,j\leq n, 0\leq \ell \leq s_{i,j}-\varepsilon_{i,j}\}$ is linearly independent and 
$$
\text{span}\{\tilde{E}_\ell^{(i,j)} \mid 1\leq i,j\leq n, 0\leq \ell\leq s_{i,j}-\varepsilon_{i,j}  \}
$$ 
is closed under ordinary multiplication. 
Since 
$$
\text{span}\{\tilde{A}_\ell^{(i,j)} \mid 1\leq i,j\leq n, \ell\in A(Y_i,Y_j)  \}=\text{span}\{\tilde{E}_\ell^{(i,j)} \mid 1\leq i,j\leq n, 0\leq \ell\leq s_{i,j}-\varepsilon_{i,j}  \}
$$ 
holds, $(\bigcup_{i=1}^n Y_i, \{\tilde{R}_\ell^{(i,j)} \mid 1\leq i,j\leq n,\varepsilon_{i,j}\leq \ell \leq s_{i,j}\}) $ is a coherent configuration.
Krein numbers of the coherent configuration are positive scalar multiple of those for the $Q$-polynomial coherent configuration, so Proposition~\ref{prop:Q-poly}(3) is satisfied.
\end{proof}

\section{Future works}\label{sec:open}
In the present paper, we introduce the $Q$-polynomial property for coherent configurations. 
The parameters of the coherent configurations are studied in the same manner as association schemes and several examples are obtained from Delsarte designs in $Q$-polynomial association schemes and spherical designs. 
We list the related problems in this context. 
\begin{problem}
\begin{enumerate}
\item 
Can we develop design theory in $Q$-polynomial coherent configurations?  
For bounds for subsets in coherent configurations, see \cite{H,HW2014JAC}  
\item In \cite{BB2010}, it was shown that Euclidean designs with certain property have the structure of coherent configurations. Are these coherent configurations $Q$-polynomial? See \cite{NS} for Euclidean designs. 
\item Can we obtain Euclidean designs from coherent configurations? If so, can we determine the strength as Euclidean designs from parameters of $Q$-polynomial coherent configurations? See \cite{SJCD} for spherical designs obtained from $Q$-polynomial schemes.  
\item The absolute bound for symmetric association schemes was shown in \cite[Theorems~4.8, 4.9]{BI} and examples attaining the inequality in \cite[Theorem~4.9]{BI} are tight spherical designs. 
On the other hand, the absolute bound for coherent configurations was shown in \cite{HW2014LAA}. Are examples of coherent configurations attaining the absolute bound related to tight Euclidean designs?    
\item In \cite{ST}, the cross-intersection theorem is stated in coherent configurations related to Grassmann schemes. Can we deal with the cross-intersection theorem in $Q$-polynomial coherent configurations whose fibers are distance regular graphs? 
\end{enumerate}
\end{problem}

\section*{Acknowledgments}
The author would like to thank Hajime Tanaka for valuable comments, especially suggesting Section~\ref{subsec:pq} and Theorem~\ref{thm:Q-polyccDel2}, and encouraging him for a decade.   
The author is supported by JSPS KAKENHI Grant Number 18K03395 and 20K03527.


\appendix
\section{Appendix: Parameters}\label{sec:a}

\begin{proposition}\label{1}
Let $\mathcal{C}$ be a coherent configuration such that each fiber is a symmetric association scheme and there exists a basis $\{E_\ell^{(i,j)}\mid i,j\in\Omega, 0\leq \ell\leq \tilde{r}_{i,j}\}$ of $\mathcal{A}$ satisfying (B1)-(B4). 
Then the following $(1)-(6)$ hold:
\begin{enumerate}
\item $p_{0,m,n}^{(i,i,h)}=\delta_{m,n}$,
\item $p_{\ell,0,n}^{(i,j,j)}=\delta_{\ell,n}$,
\item $p_{\ell,m,0}^{(i,j,i)}=\delta_{\ell,m}k_\ell^{(i,j)}$,
\item $p_{\ell,m,n}^{(i,j,h)}=p_{m,\ell,n}^{(h,j,i)}$,
\item $\sum\limits_{m=\varepsilon_{j,h}}^{r_{j,h}}p_{\ell,m,n}^{(i,j,h)}=k_\ell^{(i,j)}$,
\item $|X_i|k_n^{(i,j)}p_{\ell,m,n}^{(i,h,j)}=|X_j|k_m^{(j,h)}p_{n,\ell,m}^{(j,i,h)}=|X_h|k_\ell^{(h,i)}p_{m,n,\ell}^{(h,j,i)}$.
\end{enumerate}
\end{proposition}
\begin{proof}
$(1)-(3)$ are obvious from definition of intersection numbers.

(4): Count the number of elements in $\{z\in X_j \mid (x,z)\in R_\ell^{(i,j)},(z,y)\in R_m^{(j,h)}\}$ for $(x,y)\in R_n^{(i,h)}$,
\begin{align*}
p_{\ell,m,n}^{(i,j,h)}&=|\{z\in X_j \mid (x,z)\in R_\ell^{(i,j)},(z,y)\in R_m^{(j,h)}\}|\\
&=|\{z\in X_j \mid (y,z)\in R_m^{(h,j)},(z,x)\in R_\ell^{(j,i)}\}|\\
&=p_{m,\ell,n}^{(h,j,i)}.
\end{align*}
This proves $(4)$.

(5): Count the number of elements in $\{z\in X_j \mid (x,z)\in R_\ell^{(i,j)}\}$ for $(x,y)\in R_n^{(i,h)}$,
\begin{align*}\sum\limits_{m=\varepsilon_{j,h}}^{r_{j,h}}p_{\ell,m,n}^{(i,j,h)}&=|\bigcup_{m=\varepsilon_{j,h}}^{r_{j,h}}\{z\in X_j \mid (x,z)\in R_\ell^{(i,j)},(z,y)\in R_m^{(j,h)}\}|\\
&=|\{z\in X_j \mid (x,z)\in R_\ell^{(i,j)}\}|\\
&=k_\ell^{(i,j)}.
\end{align*}
This proves $(5)$.

(6): Count the number of element in $X_{\ell,m,n}^{(i,j,h)}=\{(x,y,z)\in X_i\times X_j\times X_h \mid (x,y)\in R_n^{(i,j)},(y,z)\in R_m^{(j,h)},(z,x)\in R_\ell^{(h,i)}\}$,
\begin{align*}
|X_{\ell,m,n}^{(i,j,h)}|&=|\bigcup\limits_{(x,y)\in R_n^{(i,j)}}\{(x,y,z)\mid (x,z)\in R_\ell^{(i,h)},(z,y)\in R_m^{(h,j)}\}|=|X_i|k_n^{(i,j)}p_{\ell,m,n}^{(i,h,j)}\\
&=|\bigcup\limits_{(y,z)\in R_m^{(j,h)}}\{(x,y,z)\mid (y,x)\in R_n^{(j,i)},(x,z)\in R_\ell^{(i,h)}\}|=|X_j|k_m^{(j,h)}p_{n,\ell,m}^{(j,i,h)}\\
&=|\bigcup\limits_{(z,x)\in R_\ell^{(h,i)}}\{(x,y,z)\mid (z,y)\in R_m^{(h,j)},(y,x)\in R_n^{(j,i)}\}|=|X_h|k_\ell^{(h,i)}p_{m,n,\ell}^{(h,j,i)}. 
\end{align*}
This proves $(6)$.
\end{proof}

\begin{proposition}\label{6}
Let $\mathcal{C}$ be a coherent configuration such that each fiber is a symmetric association scheme and there exists a basis $\{E_\ell^{(i,j)}\mid i,j\in\Omega, 0\leq \ell\leq \tilde{r}_{i,j}\}$ of $\mathcal{A}$ satisfying (B1)-(B4).
Then the following $(1)-(6)$ hold:
\begin{enumerate}
\item $q_{0,m,n}^{(i,j)}=\delta_{m,n}$,
\item $q_{\ell,0,n}^{(i,j)}=\delta_{\ell,n}$,
\item $q_{\ell,m,0}^{(i,j)}=\delta_{\ell,m}m_\ell^{(i,j)}$,
\item $q_{\ell,m,n}^{(i,j)}=q_{m,\ell,n}^{(i,j)}$,
\item $q_{\ell,m,n}^{(i,j)}=q_{\ell,m,n}^{(j,i)}$,
\item $m_n^{(i,j)}q_{\ell,\ell',n}^{(i,j)}=m_{\ell'}^{(i,j)}q_{n,\ell,\ell'}^{(i,j)}=m_\ell^{(i,j)}q_{\ell',n,\ell}^{(i,j)}$,
\item $\sum\limits_{\alpha=0}^{\tilde{r}_{i,j}}q_{m,n,\alpha}^{(i,j)}q_{\ell,\alpha,\beta}^{(i,j)}=\sum\limits_{\alpha=0}^{\tilde{r}_{i,j}}q_{\ell,m,\alpha}^{(i,j)}q_{n,\alpha,\beta}^{(i,j)}$.
\end{enumerate}
\end{proposition}
\begin{proof}
(1): On the one hand, by $E_0^{(i,j)}=\frac{1}{\sqrt{|X_i||X_j|}}J_{|X_i|,|X_j|}$, $E_0^{(i,j)}\circ E_m^{(i,j)}=\frac{1}{\sqrt{|X_i||X_j|}}E_m^{(i,j)}$ holds. 
On the other hand, by the definition of Krein numbers,  $E_0^{(i,j)}\circ E_m^{(i,j)}=\frac{1}{\sqrt{|X_i||X_j|}}\sum\limits_{n=0}^{\tilde{r}_{i,j}}q_{0,m,n}^{(i,j)}E_n^{(i,j)}$. 
Since $E_n^{(i,j)}$ ($n\in \{\varepsilon_{i,j},\ldots,\tilde{r}_{i,j}\}$) is a basis of $\mathcal{A}_{i,j}$, comparing these equalities proves $(1)$.

(2) is proved similarly as $(1)$. 

(3): Applying \text{tr} to $q_{\ell,m,0}^{(i,j)}E_0^{(i,i)}=\sqrt{|X_i||X_j|}(E_\ell^{(i,j)}\circ E_m^{(i,j)})E_0^{(j,i)}$,
\begin{align*}
q_{\ell,m,0}^{(i,j)}&=\text{tr}(\sqrt{|X_i||X_j|}(E_\ell^{(i,j)}\circ E_m^{(i,j)})E_0^{(j,i)})\\
&=\sqrt{|X_i||X_j|}\tau(E_\ell^{(i,j)}\circ E_m^{(i,j)}\circ E_0^{(i,j)})\\
&=\tau(E_\ell^{(i,j)}\circ E_m^{(i,j)})\\
&=\text{tr}(E_\ell^{(i,j)}E_m^{(j,i)})\\
&=\delta_{\ell,m}\text{tr}(E_\ell^{(i,i)})\\
&=\delta_{\ell,m}m_\ell^{(i,i)}. 
\end{align*}
Since 
$$
m_\ell^{(i,i)}=\mathrm{rank}E_\ell^{(i,i)}=\mathrm{rank}E_\ell^{(i,j)}E_\ell^{(j,i)}=\mathrm{rank}E_\ell^{(i,j)}=m_\ell^{(i,j)},
$$ 
$(3)$ holds.

(4) follows from $E_\ell^{(i,j)}\circ E_m^{(i,j)}=E_m^{(i,j)}\circ E_\ell^{(i,j)}$ and (5) follows from taking the transpose of $E_\ell^{(i,j)}\circ E_m^{(i,j)}=\frac{1}{\sqrt{|X_i||X_j|}}\sum\limits_{n=0}^{\tilde{r}_{i,j}}q_{\ell,m,n}^{(i,j)}E_n^{(i,j)}$ and ${E_n^{(i,j)}}^\top=E_n^{(j,i)}$.

(6): Applying $\tau$ to $\sqrt{|X_i||X_j|}E_\ell^{(i,j)}\circ E_{\ell'}^{(i,j)}\circ E_n^{(i,j)}$,
\begin{align*}
\sqrt{|X_i||X_j|}\tau(E_\ell^{(i,j)}\circ E_{\ell'}^{(i,j)}\circ E_n^{(i,j)})&=\sqrt{|X_i||X_j|}\text{tr}((E_\ell^{(i,j)}\circ E_{\ell'}^{(i,j)})E_n^{(i,j)})\\
&=\text{tr}(q_{\ell,\ell',n}^{(i,j)}E_n^{(i,j)})\\
&=q_{\ell,\ell',n}^{(i,j)}m_n^{(i,j)}. 
\end{align*}
Further by $E_\ell^{(i,j)}\circ E_{\ell'}^{(i,j)}\circ E_n^{(i,j)}=E_n^{(i,j)}\circ E_\ell^{(i,j)}\circ E_m^{(i,j)}=E_{\ell'}^{(i,j)}\circ E_n^{(i,j)}\circ E_\ell^{(i,j)}$,
$(5)$ holds.

(7): In the equation $E_\ell^{(i,j)}\circ (E_m^{(i,j)}\circ E_n^{(i,j)})=(E_\ell^{(i,j)}\circ E_m^{(i,j)})\circ E_n^{(i,j)}$, the left hand side is
\begin{align*}
E_\ell^{(i,j)}\circ (E_m^{(i,j)}\circ E_n^{(i,j)})&=E_\ell^{(i,j)}\circ(\frac{1}{\sqrt{|X_i||X_j|}}\sum\limits_{\alpha=0}^{\tilde{r}_{i,j}}q_{m,n,\alpha}^{(i,j)}E_{\alpha}^{(i,j)})\\
&=\frac{1}{|X_i||X_j|}\sum\limits_{\beta=0}^{\tilde{r}_{i,j}}(\sum\limits_{\alpha=0}^{\tilde{r}_{i,j}} q_{m,n,\alpha}^{(i,j)}q_{\ell,\alpha,\beta}^{(i,j)})E_\beta^{(i,j)}
\end{align*}
and right hand side is
\begin{align*}
(E_\ell^{(i,j)}\circ E_m^{(i,j)})\circ E_n^{(i,j)}&=(\frac{1}{\sqrt{|X_i||X_j|}}\sum\limits_{\alpha=0}^{\tilde{r}_{i,j}}q_{\ell,m,\alpha}^{(i,j)}E_{\alpha}^{(i,j)})\circ E_n^{(i,j)}\\
&=\frac{1}{|X_i||X_j|}\sum\limits_{\beta=0}^{\tilde{r}_{i,j}}(\sum\limits_{\alpha=0}^{\tilde{r}_{i,j}} q_{\ell,m,\alpha}^{(i,j)}q_{n,\alpha,\beta}^{(i,j)})E_\beta^{(i,j)}. 
\end{align*}
Comparing the coefficient of $E_\beta^{(i,j)}$ yields the desired equality. 
\end{proof}

For a matrix $A$, let $\tau(A)$ be the sum of the entries of $A$. 
\begin{proposition}\label{2}
Let $\mathcal{C}$ be a coherent configuration such that each fiber is a symmetric association scheme and there exists a basis $\{E_\ell^{(i,j)}\mid i,j\in\Omega, 0\leq \ell\leq \tilde{r}_{i,j}\}$ of $\mathcal{A}$ satisfying (B1)-(B4). 
Then the following $(1),(2)$ hold:
\begin{enumerate}
\item $\tau(A_\ell^{(i,j)})=|X_i|k_\ell^{(i,j)}$,
\item $\tau(E_\ell^{(i,j)})=\sqrt{|X_i||X_j|}\delta_{\ell,0}$. 
\end{enumerate}
\end{proposition}
\begin{proof}(1) is proved as  
\begin{align*}
\tau(A_\ell^{(i,j)})=|\{(x,y)\in X_i\times X_j \mid (x,y)\in R_\ell^{(i,j)}\}|
=|X_i|k_\ell^{(i,j)}. 
\end{align*}

(2): By $E_0^{(i,j)}=\frac{1}{\sqrt{|X_i||X_j|}}J_{|X_i|,|X_j|}$, 
\begin{align*}
\tau(E_\ell^{(i,j)})&=\tau(E_\ell^{(i,j)}\circ J_{|X_i|,|X_j|})\\ 
&=\sqrt{|X_i||X_j|}\tau(E_\ell^{(i,j)}\circ E_0^{(i,j)})\\
&=\sqrt{|X_i||X_j|}\text{tr}(E_\ell^{(i,j)}E_0^{(j,i)})\\
&=\sqrt{|X_i||X_j|}\delta_{\ell,0}\text{tr}(E_0^{(i,j)})\\
&=\sqrt{|X_i||X_j|}\delta_{\ell,0}. 
\end{align*}
This proves $(2)$.
\end{proof}

\begin{proposition}\label{3}
Let $\mathcal{C}$ be a coherent configuration such that each fiber is a symmetric association scheme and there exists a basis $\{E_\ell^{(i,j)}\mid i,j\in\Omega, 0\leq \ell\leq \tilde{r}_{i,j}\}$ of $\mathcal{A}$ satisfying (B1)-(B4).
Then the following $(1),(2)$ hold:
\begin{enumerate}
\item $p_\ell^{(i,j)}(0)=\sqrt{\frac{|X_i|}{|X_j|}}k_\ell^{(i,j)}$,
\item $q_0^{(i,j)}(m)=1$.
\end{enumerate}
\end{proposition}
\begin{proof}
(1): Apply $\tau$ to $A_\ell^{(i,j)}=\sum\limits_{m=0}^{\tilde{r}_{i,j}}p_\ell^{(i,j)}(m)E_m^{(i,j)}$ and use Proposition~\ref{2}$(1),(2)$ to obtain
\begin{align*}
|X_i|k_\ell^{(i,j)}=\tau(A_\ell^{(i,j)})&=\sum\limits_{m=0}^{\tilde{r}_{i,j}}p_\ell^{(i,j)}(m)\tau(E_m^{(i,j)})=\sqrt{|X_i||X_j|}p_\ell^{(i,j)}(0). 
\end{align*}
Dividing by $\sqrt{|X_i||X_j|}$, we obtain $(1)$.

(2): By the definition of $E_0^{(i,j)}$,
\begin{align*}
E_0^{(i,j)}&=\frac{1}{\sqrt{|X_i|||X_j|}}J_{|X_i|,|X_j|}=\frac{1}{\sqrt{|X_i||X_j|}}\sum\limits_{m=\varepsilon_{i,j}}^{r_{i,j}}A_m^{(i,j)}.
\end{align*}
Hence we obtain $q_0^{(i,j)}(m)=1$.
\end{proof}

\begin{proposition}\label{4}
Let $\mathcal{C}$ be a coherent configuration such that each fiber is a symmetric association scheme and there exists a basis $\{E_\ell^{(i,j)}\mid i,j\in\Omega, 0\leq \ell\leq \tilde{r}_{i,j}\}$ of $\mathcal{A}$ satisfying (B1)-(B4).
Then the following $(1)-(3)$ hold: 
\begin{enumerate}
\item $\frac{q_h^{(i,j)}(\ell)}{\sqrt{|X_j|}m_h^{(i,j)}}=\frac{p_\ell^{(i,j)}(h)}{\sqrt{|X_i|}k_\ell^{(i,j)}}$,
\item $\sum\limits_{\nu=\varepsilon_{i,j}}^{r_{i,j}}\frac{1}{k_{\nu}^{(i,j)}}p_{\nu}^{(i,j)}(h)p_{\nu}^{(i,j)}(\ell)=\frac{|X_i|\delta_{h,\ell}}{m_\ell^{(i,j)}}$,
\item $\sum\limits_{\nu=0}^{\tilde{r}_{i,j}}m_{\nu}^{(i,j)}p_h^{(i,j)}(\nu)p_\ell^{(i,j)}(\nu)=|X_i|k_\ell^{(i,j)}\delta_{h,\ell}$.
\end{enumerate}
\end{proposition}
\begin{proof}
(1): Applying $\tau$ to $E_h^{(i,j)}\circ A_\ell^{(i,j)}=\frac{1}{\sqrt{|X_i||X_j|}}q_h^{(i,j)}(\ell)A_\ell^{(i,j)}$, we obtain the following: 
the left hand side yields
\begin{align*}
\tau(E_h^{(i,j)}\circ A_\ell^{(i,j)})&=\text{tr}(E_h^{(i,j)}A_\ell^{(j,i)})\\
&=\text{tr}(p_\ell^{(j,i)}(h)E_h^{(i,i)})\\
&=p_\ell^{(j,i)}(h)m_h^{(i,i)},
\end{align*}
and on the other hand, the right hand side yields 
\begin{align*}
\tau(\frac{1}{\sqrt{|X_i||X_j|}}q_h^{(i,j)}(\ell)A_\ell^{(i,j)})&=\frac{1}{\sqrt{|X_i||X_j|}}q_h^{(i,j)}(\ell)\tau(A_\ell^{(i,j)})\\
&=\sqrt{\frac{|X_i|}{|X_j|}}q_h^{(i,j)}(\ell)k_\ell^{(i,j)}.
\end{align*}
By $m_h^{(i,i)}=m_h^{(i,j)}$ and $p_\ell^{(i,j)}(h)=p_\ell^{(j,i)}(h)$, $(1)$ holds.

(2), (3): By $(1)$ and $P^{(i,j)}Q^{(i,j)}=Q^{(i,j)}P^{(i,j)}=\sqrt{|X_i||X_j|}I$, $(2)$ and $(3)$ hold.
\end{proof}

\begin{proposition}\label{5}
Let $\mathcal{C}$ be a coherent configuration such that each fiber is a symmetric association scheme and there exists a basis $\{E_\ell^{(i,j)}\mid i,j\in\Omega, 0\leq \ell\leq \tilde{r}_{i,j}\}$ of $\mathcal{A}$ satisfying (B1)-(B4).
Then the following $(1),(2)$ hold:
\begin{enumerate}
\item $q_{\ell,\ell',n}^{(i,j)}=\frac{\sqrt{|X_i||X_j|}m_\ell^{(i,j)}m_{\ell'}^{(i,j)}}{|X_i|^2}\sum\limits_{\nu=\varepsilon_{i,j}}^{r_{i,j}}\frac{1}{{k_{\nu}^{(i,j)}}^2}p_{\nu}^{(i,j)}(\ell)p_{\nu}^{(i,j)}(\ell')p_{\nu}^{(i,j)}(n)$,
\item $p_{\ell,m,n}^{(i,j,h)}=\frac{k_\ell^{(i,j)}k_m^{(j,h)}}{|X_h|}\sum\limits_{\nu=0}^{\min\{\tilde{r}_{i,j},\tilde{r}_{j,h},\tilde{r}_{h,i}\}}\frac{1}{{m_{\nu}^{(i,i)}}^2}q_{\nu}^{(i,j)}(\ell)q_{\nu}^{(j,h)}(m)q_{\nu}^{(h,i)}(n)$.
\end{enumerate}
\end{proposition}
\begin{proof}
(1): Applying \text{tr} to $q_{\ell,\ell',n}^{(i,j)}E_n^{(i,i)}=\sqrt{|X_i|||X_j|}(E_\ell^{(i,j)}\circ E_{\ell'}^{(i,j)})E_n^{(j,i)}$,
\begin{align*}
q_{\ell,\ell',n}^{(i,j)}m_n^{(i,i)}&=\sqrt{|X_i|||X_j|}\text{tr}((E_\ell^{(i,j)}\circ E_{\ell'}^{(i,j)})E_n^{(j,i)})\\
&=\sqrt{|X_i|||X_j|}\tau(E_\ell^{(i,j)}\circ E_{\ell'}^{(i,j)}\circ  E_n^{(i,j)})\\
&=\frac{1}{|X_i||X_j|}\sum\limits_{\nu=\varepsilon_{i,j}}^{r_{i,j}}q_\ell^{(i,j)}(\nu)q_{\ell'}^{(i,j)}(\nu)q_n^{(i,j)}(\nu)\tau(A_{\nu}^{(i,j)})\\
&=\frac{1}{|X_i||X_j|}\sum\limits_{\nu=\varepsilon_{i,j}}^{r_{i,j}}\left(\sqrt{\frac{|X_j|}{|X_i|}}\right)^3\frac{m_\ell^{(i,j)}m_{\ell'}^{(i,j)}m_n^{(i,j)}}{{k_{\nu}^{(i,j)}}^3}p_\ell^{(i,j)}(\nu)p_{\ell'}^{(i,j)}(\nu)p_n^{(i,j)}(\nu)|X_i|k_{\nu}^{(i,j)}\\
&=\frac{\sqrt{|X_i||X_j|}m_\ell^{(i,j)}m_{\ell'}^{(i,j)}m_n^{(i,j)}}{|X_i|^2}\sum\limits_{\nu=\varepsilon_{i,j}}^{r_{i,j}}\frac{1}{{k_{\nu}^{(i,j)}}^2}q_\ell^{(i,j)}(\nu)q_{\ell'}^{(i,j)}(\nu)q_n^{(i,j)}(\nu).
\end{align*}
Dividing by $m_n^{(i,i)}=m_n^{(i,j)}$, we obtain $(1)$.

(2): Applying $\tau$ to $p_{\ell,m,n}^{(i,j,h)}A_n^{(i,h)}=(A_\ell^{(i,j)}A_m^{(j,h)})\circ A_n^{(i,h)}$,
\begin{align*}
p_{\ell,m,n}^{(i,j,h)}|X_i|k_n^{(i,h)}&=\tau((A_\ell^{(i,j)}A_m^{(j,h)})\circ A_n^{(i,h)})\\
&=\text{tr}(A_\ell^{(i,j)}A_m^{(j,h)}A_n^{(h,i)})\\
&=\sum\limits_{\nu=0}^{\min\{\tilde{r}_{i,j},\tilde{r}_{j,h},\tilde{r}_{h,i}\}}p_\ell^{(i,j)}(\nu)p_m^{(j,h)}(\nu)p_n^{(h,i)}(\nu)\text{tr}(E_{\nu}^{(i,i)})\\
&=\sum\limits_{\nu=0}^{\min\{\tilde{r}_{i,j},\tilde{r}_{j,h},\tilde{r}_{h,i}\}}\frac{k_\ell^{(i,j)}k_m^{(j,h)}k_n^{(h,i)}}{m_{\nu}^{(i,j)}m_{\nu}^{(j,h)}m_{\nu}^{(h,i)}}q_{\nu}^{(i,j)}(\ell)q_{\nu}^{(j,h)}(m)q_{\nu}^{(h,i)}(n)m_{\nu}^{(i,i)}\\
&=k_\ell^{(i,j)}k_m^{(j,h)}k_n^{(h,i)}\sum\limits_{\nu=0}^{\min\{\tilde{r}_{i,j},\tilde{r}_{j,h},\tilde{r}_{h,i}\}}\frac{1}{{m_{\nu}^{(i,i)}}^2}q_{\nu}^{(i,j)}(\ell)q_{\nu}^{(j,h)}(m)q_{\nu}^{(h,i)}(n). 
\end{align*}
Dividing by $|X_i|k_n^{(i,h)}=|X_k|k_n^{(h,i)}$, we obtain $(2)$. 
\end{proof}

\end{document}